\documentclass{article}

\usepackage[german,english]{babel}

\usepackage{amsmath, latexsym, amsfonts, amssymb, amsthm, amscd, color}
\usepackage[parfill]{parskip}
\usepackage[utf8]{inputenc}
\usepackage{mathtools}
\usepackage{mathrsfs}
\usepackage{enumitem}
\setenumerate[0]{label=\textbf{(\roman*)}}
\newtheorem*{theorem1}{Theorem}

\usepackage{titlesec}
\usepackage{tocloft}

\newcommand{\R}{\mathbb{R}}
\newcommand{\N}{\mathbb{N}}
\newcommand{\E}{\mathbb{E}}
\renewcommand{\P}{\mathbb{P}}
\newcommand{\1}{\textbf 1}
\newcommand{\dd}{\,\mathrm{d}} %upright d, with preceding space
\newcommand{\ddd}{\mathrm{d}} %upright d
\newcommand{\ee}{\mathrm{e}} %upright e
\newcommand{\abs}[1]{\lvert#1\rvert}
\newcommand{\eps}{\varepsilon}

\theoremstyle{plain}
\newtheorem{theorem}{Theorem}[section]
\newtheorem{definition}[theorem]{Definition}
\newtheorem*{theorem*}{Theorem}
\newtheorem{proposition}[theorem]{Proposition}
\newtheorem{lemma}[theorem]{Lemma}
\newtheorem*{lemma*}{Lemma}
\newtheorem{corollary}[theorem]{Corollary}
\newtheorem{example}[theorem]{Example}

\usepackage{tikz}

\makeatletter
\def\namedlabel#1#2{\begingroup
    #2%
    \def\@currentlabel{#2}%
    \phantomsection\label{#1}\endgroup
}

\allowdisplaybreaks

\titleformat{\chapter}[hang]{\normalfont\bfseries}{\huge \thechapter}{12pt}{\huge}

\titleformat{\section}[hang]{\normalfont\bfseries}{\Large \thesection}{8pt}{\Large}
\titleformat{\subsection}[hang]{\normalfont\bfseries}{\large \thesubsection}{8pt}{\large}
\titleformat{\subsubsection}[hang]{\normalfont\bfseries}{\hspace{-0.75em}}{8pt}{\large}

\usepackage[hidelinks]{hyperref}

\hypersetup{
    pdftitle={General path integrals and stable SDEs},
    bookmarksnumbered=true,     
    bookmarksopen=true,         
    bookmarksopenlevel=1,
}

\usepackage{authblk}

\author[$\dagger$]{Sam Baguley}
\author[$\dagger$]{Leif D\"oring}
\author[$\star$]{Andreas Kyprianou}

\affil[$\dagger$]{University of Mannheim, Institute of Mathematics, 68131 Mannheim, Germany.}
\affil[$\star$]{University of Bath, Department of Mathematical Sciences, Bath, BA2 7AY.}

\title{General path integrals and stable SDEs}

\begin{document}
%------------------------------------------------------------------------------
\maketitle
%\tableofcontents

\begin{abstract}
The theory of one-dimensional stochastic differential equations driven by Brownian motion is classical and has been largely understood for several decades. For stochastic differential equations with jumps the picture is still incomplete, and even some of the most basic questions are only partially understood. In the present article we study existence and uniqueness of weak solutions to
\begin{align*}
	\dd Z_t=\sigma(Z_{t-})\dd X_t
\end{align*}	
driven by a (symmetric) $\alpha$-stable L\'evy process, in the spirit of the classical Engelbert-Schmidt time-change approach. Extending and completing results of Zanzotto we derive a complete characterisation for existence und uniqueness of weak solutions for $\alpha\in(0,1)$. Our approach is not based on classical stochastic calculus arguments but on the general theory of Markov processes. We proof integral tests for finiteness of path integrals under minimal assumptions.
%the general theory of Markov processes instead.
\smallskip

\smallskip

\end{abstract}

%------------------------------------------------------------------------------
%------------------------------------------------------------------------------
% SECTION 1: INTRODUCTION
\section{Introduction}

Itô diffusions are solutions of stochastic differential equations driven by a Brownian motion, and their study as stochastic processes in their own right has a rich history dating back to the foundational works of Feller \cite{Feller1}, \cite{Feller2} in the 1950s. Itô diffusions continue to motivate theoretical research, due in some part to their many applications, and this sustained interest has also given rise to various generalisations. The present article is concerned with solutions of stochastic differential equations driven not by a Brownian motion but a Lévy process. %, where the intensity measure of the jumps is $dt\otimes \nu(dx)$, where the L\'evy measure $\nu$ has support in $\R\setminus\{0\}$ and satisfies
%\begin{align*}
%	\int_\R (1 \wedge x^2) \, \nu(\dd x) < \infty.
%\end{align*}
Our objective is to study solutions to the one-dimensional driftless stochastic differential equation
\begin{align}\label{eq1}
	\ddd Z_t = \sigma(Z_{t-}) \dd X_t, \quad Z_0 = z\in\R,
\end{align}
where $X$ is a Lévy process and $\sigma$ is a non-negative, measurable function on $\R$. Solution processes $Z$ are distinguished in contrast to continuous Itô diffusions by the presence of discontinuities and are therefore known as jump diffusions. In the present article we always assume the driving L\'evy process $X$ is strictly stable, that is, the scaling property
\[
(cX_{c^{-\alpha}t}, t\geq 0) \text{ under } \mathbb{P}_x \text{ is equal in law to } (X_t, t\geq 0) \text{ under }\mathbb{P}_{cx},
\]
holds for some $\alpha\in (0,2]$. The strictly stable process of index $\alpha = 2$ is simply the Brownian motion. %The strictly stable process of index $\alpha = 1$ is the only one which can have a non-zero linear drift, but we shall assume that this drift is zero; this process is called the Cauchy process.

In 1981 Engelbert and Schmidt \cite{ES} proved a zero-one law for finiteness of so-called finite-time path integrals $\int_0^t f(B_s)\dd s$ of the Brownian motion. In the same paper they made use of their zero-one law and a time-change representation of Brownian SDEs to comprehensively address the question of weak existence and uniqueness of solutions to \eqref{eq1}, see Theorem 4 of \cite{ES} or Theorems 3.4.2 and 3.4.6 of Karatzas and Shreve \cite{KS}. Almost two decades later Zanzotto \cite{Zanzotto97} extended Engelbert and Schmidt's results to stable processes of index $\alpha\in(1,2]$. Proving the time-change representation for stable SDEs is significantly harder than for Brownian SDEs, because it is no longer possible to argue using quadratic variation. Despite this difficulty, Zanzotto's result is nearly identical in form to that for the Brownian motion, see Theorem 2.32 of his paper. A few years later Zanzotto generalised his result further, see Theorems 2.2 and 2.6 of \cite{Zanzotto02}, but was still restricted to $\alpha\in(1,2]$ for precise results. Harder results for SDEs concern strong existence and pathwise uniqueness. Following seminal work of Yamada and Watanabe \cite{YW} for the Brownian case, strong existence and pathwise uniqueness has been studied in many variations. The classical setup gives positive results for Lipschitz continuous drift and $\frac 1 2$-H\"older continuous noise coefficient. Extensions to stable SDEs are surprisingly recent, mostly motivated by the study of continuous state branching processes. For spectrally positive stable processes (only positive jumps) and $\alpha\in (1,2)$ the Yamada/Watanabe argument was generalised by Li and Mytnik \cite{LM} leading to Lipschitz continuous drift and $(1-\frac 1 \alpha$)-H\"older continuous noise coefficient. The symmetric case was dealt with by Bass \cite{Bass} leading to $\frac 1 \alpha$-H\"older continuous noise coefficient. Since $\frac 1 \alpha>1$ for $\alpha\in (0,1)$ it comes as no surprise that results must be structurally different for small $\alpha$. Partial results for $\alpha\in (0,1)$ on the pathwise uniqueness problem appeared in the past two decades (see for instance Bass/Burdzy/Chen \cite{BBC} for counter example). To the best of our knowledge proofs for sharp conditions are still unknown. The aim of the present article is not to solve the strong existence/uniqueness problem, but to give a complete solution to the weak existence/uniqueness problem in the spirit of Engelbert and Schmidt's results.

The works of Engelbert/Schmidt and Zanzotto together answer questions for weak solutions to \eqref{eq1} in the case of $\alpha\in(1,2]$. In all of them the time-change method is used to reduce the SDE to a time-change of the driving process $X$ using the time-change $\int_0^t \sigma(X_s)^{-\alpha}\dd s$, which is usually called a path integral (or perpetual integral). The crucial tool for analysing those path integrals is the occupation time formula and local time for $X$. This is the barrier to results for $\alpha\in (0,1]$: while the time-change representations of Zanzotto continue to hold, there is no local time to work with. The challenge we tackle in the present article is the study of finite- and infinite-time path integrals of the form
$$
	I^f_t = \int_0^t f(X_s)\dd s, \qquad t\in(0,\infty],
$$
where $f$ is a measurable function taking values in $[0,+\infty]$. The value $+\infty$ for $f$ is crucial as it corresponds to the zeros of $\sigma$. What we seek is a collection of `integral tests', which are statements tying the law of a path integral to finiteness of a related deterministic integral. Path integrals have attracted some interested in recent years, mostly for restricted classes of integrands (such as $f$ being locally integrable or continuous) and processes (such as L\'evy processes or diffusion processes) leading to very clean results. For $X$ a Brownian motion with positive drift, Salminen and Yor \cite{SY1, SY2} obtained results via the Ray-Knight theorem,
% Jeulin’s lemma, and Khashminkii’s lemma, 
and similar results for spectrally negative Lévy processes can be found in Koshnevisan, Salminen, and Yor \cite{KSY}. The recent article of Kolb and Savov \cite{KolbSavov} has the most up to date results for L\'evy processes. The technique we use only requires minimal assumptions; i.e. $f$ measurable and $X$ a  standard Markov process. Allowing $f$ to be infinite changes the picture completely as known $0$-$1$-laws for path integrals (see for instance Kyprianou and D\"oring \cite{DK1}) fail immediately for transient $X$ (take $f$ an infinite indicator on a set which $X$ visits with probability in $(0,1)$). Hence, it is clear that for a general theorem sets which hit with probability less than one must appear in the statements of results. Such sets are called avoidable; we call their complements supportive, and these sets will play a fundamental role in this work. Assuming only that $X$ is a standard Markov process on a general state space $E$ with potential measure $U$ we will prove the following theorem on path integrals.
\begin{theorem1}
Let $f:E \to [0,+\infty]$ be measurable and $X$ a standard Markov process on $E$ with (possibly infinite) life time $\zeta$. Let $z\in E$. 
%Then the following equivalence is in place:
%\begin{align*}
%	&\quad \P_z \Big( \int_0^\zeta f(X_s)\dd s < \infty\Big)>0\\
%	&\Longleftrightarrow \quad \text{There is a } \P_z\text{-supportive set } M \text{ such that } \int_M f(x) U(z,\mathrm d x)<\infty. 
%\end{align*}
Then the following are equivalent.
\begin{enumerate}
	\item $\P_z\big( \int_0^\zeta f(X_s) \dd s < \infty\big) > 0;$
	\item The integral test $\int_{E\backslash B} f(x) \, U(z, \ddd x) < \infty
	$
	holds for a $\P_z$-avoidable set $B$.
\end{enumerate}
\end{theorem1}
We remind the reader that the integral test in (ii) says nothing but $\E_z[\int_0^\zeta \1_{B^C}(X_s)f(X_s)\ddd s]<\infty$. Thus, the path integral is finite with positive probability if and only if it has finite mean away from an avoidable set. In several situations the potential measure $U$ is explicit and avoidable sets can be described analytically to turn the integral tests into analytic statements.\smallskip

The general form of the path integral theorem allows us to deduce a couple of consequences on finite time path integrals which we will need to study the SDE \eqref{eq1} via time-change techniques. For finite time path integrals the role of avoidable sets is replaced by sets which are avoided for a positive amount of time, so-called thin sets. Let us summarise the main findings for the SDE \eqref{eq1} driven by a symmetric stable L\'evy process with $\alpha\in (0,1)$, see Section \ref{sec:SDE} for the theorems.

\begin{theorem1}
If $N(\sigma)$ denotes the zero-set of $\sigma$ and 
\begin{align}\label{O def1}
	 \mathcal{O}(\sigma,\alpha) = \Big\{x\in\R : \int_{\R\setminus B} \sigma(y)^{-\alpha} \abs{x-y}^{\alpha-1} \dd y = \infty\text{ for all }\P_x\text{-thin sets } B \Big\}
\end{align}
denotes the set of irregular points, then the following statements hold.
\begin{enumerate}
	\item For fixed $z\in\R$ there exists a non-trivial (i.e. non-constant) local weak solution if and only if $z\notin \mathcal O(\sigma,\alpha)$. 
	\item A global weak solution exists for all initial conditions $z\in\R$ if and only if $\mathcal O(\sigma,\alpha)\subseteq N(\sigma)$.
	\item A non-trivial global weak solution exist for all $z\in\R$ if and only if $\mathcal O(\sigma, \alpha)=\emptyset$.
	\item There exists a global weak solution for all $z\in\R$, each of which is unique in law, if and only if $\mathcal{O}(\sigma,\alpha) = N(\sigma)$.
\end{enumerate}
\end{theorem1}
Since thin sets of stable processes have analytic descriptions through capacities, the theorem gives a full analytic descriptions of weak solutions for the stable SDE \eqref{eq1}. The analytic description allows for the $\P_x$-thin sets to be removed in many examples such as $\sigma$ with a monotone zero at $x$.\smallbreak 

Statements (i)-(iv) are identical to the Engelbert-Schmidt/Zanzotto theorems for Brownian SDEs and stable SDEs for $\alpha\in (1,2)$, but differ in the set $\mathcal O(\sigma,\alpha)$ of irregular points, which in those cases is given by
\begin{align*}
	 \mathcal{O}(\sigma,\alpha) = \Big\{x\in\R : \int_{x} \sigma(y)^{-\alpha} \dd y = \infty \Big\}.
\end{align*}
This set is different from \eqref{O def1} in two respects. First, since complements of $\P_x$-thin sets contain balls around $x$ for $\alpha\in (1,2)$ those sets do not appear.
%, and the integral tests would become local integral tests around $x$, i.e. $\int_{x} \sigma(y)^{-\alpha} \abs{x-y}^{\alpha-1} \dd y = \infty$. 
Second, the polynomial factor in the integral tests is not present; for $\alpha \in(0,1)$ this factor has a pole, which puts a tighter restriction on the function $\sigma$.
Let us quickly give two examples that highlight the two differences.

%Before going into the proofs let us discuss the relation to the results of Engelber and Schmidt ($\alpha=2$) and Zanzotto ($\alpha\in (1,2)$). The structure of the statements (i)-(iii) is equal, the two-fold difference appears in the definition of $\mathcal O(\sigma,\alpha)$ which in their case is  $$\mathcal O(\sigma,\alpha)=\left\{ x\in\R: \int_{x-\epsilon}^{x+\epsilon} \sigma(y)^{-\alpha} \dd y=\infty\text{ for all }\epsilon>0\right\}.$$ We note that for $\alpha\in (1,2]$ all sets $\R\backslash [x-\epsilon,x+\epsilon]$ are $\P_x$-thin for the stable process and complements of $\P_x$-thin sets contain a ball around the origin. Hence, the sets $\mathcal O(\sigma,\alpha)$ of Engelbert-Schmidt/Zanzotto could be rewritten as
%\begin{align*}
%	\mathcal O(\sigma,\alpha)&=\left\{ x\in\R: \int_{\R\backslash B} \sigma(y)^{-\alpha}\dd y=\infty \text{ for all }\P_z\text{-thin sets}\right\}.
%\end{align*}
%Hence, the difference of our results to Engelbert-Schmidt/Zanzotto is an additional factor $|y|^{\alpha-1}$ in the integral test for $\alpha \in (0,1)$ and the non-triviality of thin-sets for $\alpha\in (0,1)$. To get a feeling for both let us consider two examples:

\textbf{Integral test}: To get an idea for the additional factor $|y|^{\alpha-1}$ it is instructive to check the simple example 
\begin{align}\label{power}
	dZ_t=|Z_t|^\beta dX_t,\quad Z_0=0.
\end{align}
 Since the trivial solution $Z\equiv 0$ always exists, using (i) the Engelbert-Schmidt/Zanzotto criterion implies that weak uniqueness holds for $\alpha\in (1,2]$ if and only if $\int_0 \sigma^{-\alpha} (y)\dd y=\int_0 |y|^{-\beta\alpha}\dd y=+\infty$ around $0$, which is equivalent to $\beta > \frac 1 \alpha$ and thus coincides with the Yamada-Watanabe-Bass criterion for pathwise uniqueness. Since $\sigma(x)=|x|^\beta$ is Lipschitz continuous for $\beta>1$ the Engelbert-Schmidt/Zanzotto integral test must be wrong for $\alpha\in (0,1)$. Let us check our modified integral test with the additional factor $|y|^{\alpha-1}$ but ignoring the thin sets (which we justify later in Corollary \ref{sde thm1b} as $\sigma$ has a monotone zero at $0$):
$$\int_0 \sigma(y)^{-\alpha}|y|^{\alpha-1}\dd y=\int_0 |y|^{-\beta \alpha}|y|^{\alpha-1}\dd y=\int_0 |y|^{\alpha(\beta-1)-1}\dd y=+\infty\quad \Leftrightarrow\quad \beta>1.$$
Hence, (i) implies that weak uniqueness for the SDE \eqref{power} holds if and only if $\beta>1$, the Lipschitz case. In contrast to $\alpha\in (1,2)$ the statement is independent of $\alpha$ which bears some similarity to the results of Bass/Burdy/Chen \cite{BBC}. Of course, this is an artefact of $\sigma$ being a polynomial, in the general setting our integral test is not independent of $\alpha$.

\textbf{Thin sets}: The next example is helpful for understanding how the path behavior of the driving stable process forces the thin sets to appear in the integral test. For a symmetric stable process with $\alpha\in (0,1)$ we can choose a $\P_0$-thin set $A$ with positive potential (for example a union of well-chosen small disjoint intervals accumulating at $0$, similarly to Example \ref{ex}). If we define $\sigma=\1_{A^C}$ we can construct a local weak solution by running the stable process until it hits $A$. Since $A$ is thin at $0$, the solution is non-trivial. Without the thin set $B=A$ in the integral test, statement (i) would fail as $\sigma^{-\alpha}$ equals $+\infty$ on $A$ and $A$ has positive potential. The situation is simpler for $\alpha\in (1,2]$ because in that case a set is thin at 0 if and only if its complement contains a ball around 0.\smallskip

	The present article only deals with SDEs for symmetric stable processes with $\alpha\in (0,1)$, exploiting crucially the transience. We were only able to derive partial results for the Cauchy process ($\alpha=1$) which we defer to future research.

\subsection*{Organisation of the article}
The article is organised as follows. The core of the work is carried out in Section \ref{sec:perpetualintegral}, where we present the general integral test for infinite-time-horizon path integrals $\int_0^\infty f(X_s)\dd s$ for general standard processes. By killing these processes in various ways we deduce integral tests for finite-time-horizon integral tests $\int_0^t f(X_s)\dd s$ in Sections \ref{sec:finiteperpetual} and \ref{sec:finiteperpetualtwo}. In Section \ref{sec:SDE} the results from Sections \ref{sec:perpetualintegral}, \ref{sec:finiteperpetual} and \ref{sec:finiteperpetualtwo} are translated into SDE theorems using Zanzotto's time-change techniques. 

\subsection*{Acknowledgements}
The authors wish to thank Quan Shi for many fruitful ideas and discussions regarding the work of Section \ref{sec:finiteperpetual}, and Philip Wei\ss mann and Lukas Trottner for their input and insight on Markov processes. Great thanks also to Mateusz Kwa\'sniki, whose knowledge of potential theory was invaluable at several crucial moments.

%------------------------------------------------------------------------------
%------------------------------------------------------------------------------
% SECTION 2: PERPETUAL INTEGRALS

\section{Infinite-Time Path Integral Tests}\label{sec:perpetualintegral}

\subsection{Setting}\label{sec:setting}
%------------------------------------------------------------------------------
Before we begin it is worthwhile precisely establishing the setting in which our theorems are proved, as several aspects of the proofs rely upon it. We do not assume anything out of the ordinary, and the reader familiar with the potential theory of Markov processes can skip ahead to Section \ref{mainthm}.

Take $E$ to be a locally compact Hausdorff space with a countable base, and let $\Delta$ be adjoined to $E$ as the point at infinity if $E$ is non-compact, and as an isolated point if $E$ is compact. Let $\mathcal E$ be the Borel $\sigma$-algebra on $E$, and $\mathcal E_\Delta$ the Borel $\sigma$-algebra on $E_\Delta \coloneqq E \cup  \{\Delta\}.$ Let $D$ be the space of paths $w : [0,\infty] \to E_\Delta$ such that $w(\infty)=\Delta$, and if $w(t)=\Delta$ then $w(s)=\Delta$ for all $s\ge t$.
Let ($X_t$, $t\in [0,\infty]$) be the family of coordinate maps $X_t:D \to E_\Delta$, i.e. $X_t(w) = w_t \text{ for all }t\in[0,\infty],$ and denote by $\mathcal F_t = \sigma$($X_s$, $0\le s \le t$) and $\mathcal F = \sigma$($X_s$, $s\in [0,\infty]$) the canonical filtration of $X$. Let $(\theta_t, t\in[0,\infty])$ be the family of translation maps $\theta_t: D\to D:w \mapsto (w_{t+s}, s\ge0)$, and let ($\P_x$, $x\in E_\Delta$) be a family of probability measures on $(D, \mathcal F)$ satisfying the following conditions:
\begin{enumerate}
% This def of regularity actually comes from Blumenthal an Getoor Theorem 3.6(a). It's equivalent, but from this more general formulation it can be shown that super-finite sets are measurable far more easily.
%	\item \textbf{\emph{(regularity)}} For each $t\in[0,\infty)$ and each $B\in\mathcal E$, the map $x \mapsto \P_x(X_t\in B)$ is $\mathcal E$-measurable.
	\item \emph{(regularity)} For each measurable $Y:D\to E$ and each $B\in\mathcal E$, the map $x \mapsto \P_x(Y\in B)$ is $\mathcal E_\Delta$-measurable.
% NOTE: When the $\sigma$-algebra on the second space is not explicitly named, it is assumed to be the Borel $\sigma$-algebra on $\bar \R$. See BG page 2.
	\item \emph{(normality)} $\P_x(X_0 = x) = 1$ for all $x\in E_\Delta$.
	\item \emph{(càdlàg paths)} The path functions $t\to X_t(w)$ are right continuous on $[0,\infty)$ and have left limits on $[0,\zeta)$ $\P_x$-almost surely for all $x\in E$, where the random time $\zeta(w):= \inf\{t>0 : X_t(w) = \Delta\}$ is called the \emph{lifetime} of $X$.
	\item \emph{(quasi-left-continuity)} For any sequence ${T_n}$ of $\mathcal F_t$-stopping times with limit $T$, it holds that $X_{T_n} \to X_T$ $\P_x$-almost surely on $\{T<\zeta\}$, for all $x\in E$.
	\item \emph{(strong Markov property)} For all $x\in E_\Delta$, stopping times $T$, $s\in[0,\infty]$, and bounded measurable $f$ the strong Markov property holds:
	$$
		\E_x[f(X_{T+s}) | \mathcal F_T] = \E_{X_T}[f(X_s)]\qquad \P_x\text{-almost surely.}
	$$
\end{enumerate}

The family ($\P_x$, $x\in E_\Delta$) is called a standard Markov process on state space $(E,\mathcal E)$, with cemetary state $\Delta$ and lifetime $\zeta$. It is often more convenient to refer to $X$ (with associated laws $\P_x$) as the process. The assumption that $E$ is Hausdorff ensures that compact subsets of $E$ are closed, and therefore $\mathcal E$-measurable. In the proofs that follow we shall also work with a strong Markov process (a time-changed Markov process) which may not satisfy either (ii) or (iv). %In Corollary I.8.6 of Blumenthal and Getoor \cite{BG} it was shown that for $X$ a standard Markov process, one has
%\begin{align}\label{?}l{markov}
%	\E_x[Y\circ \theta_T | \mathcal F_T] = \E_{X_T}[Y]
%\end{align}
%for every $Y\in b\mathcal F$, all stopping times $T$, and all $x\in E$. This is more versatile form of the strong Markov property (v).

 For future use we define, for a Borel set $B\in \mathcal E_\Delta$, the random times
\begin{align*}
	D_B  \coloneqq \inf\{t\ge 0 : X_t \in B\}, \quad	T_B  \coloneqq \inf\{t> 0 : X_t \in B\}, \quad	L_B  \coloneqq \sup\{t\ge 0 : X_t \in B\}.
\end{align*}
Both $D_B$ and $T_B$ are stopping times, and are called the \emph{first entry time} and \emph{first hitting time} of $B$ respectively. $L_B$ is not in general a stopping time, and is called the \emph{last exit time} of $B$. It is convention to set $\inf \emptyset = \infty$ and $\sup \emptyset = 0$. \smallskip

The \emph{(infinite time horizon) path integral} over a Markov process $X$ on state space $(E,\mathcal E)$ and non-negative $\mathcal E_\Delta$-measurable function $f:E_\Delta \to [0,\infty]$ is defined to be
$$
	I^f_\infty = \int_0^\infty f(X_s) \dd s.
$$
If $X$ has semigroup $(P_t)_{0<t<\infty}$ and lifetime $\zeta$, then the potential measure of $X$ is defined as
$$
	U(z,B) = \int_0^\infty P_t(z,B) \dd t = \E_z \Big[ \int_0^\zeta \1_{(X_t\in B)} \dd t \Big],
	\quad \text{where }P_t(z,B)=\P_z(X_t\in B),
$$
for $x\in E$ and measurable $B$, with corresponding potential operator 
$$
	Uf(z):= \int_E f(x) \, U(z,\ddd x)=\E_z\Big[\int_0^\zeta f(X_s)\dd s\Big]%=\E_z[I^f_\infty]
$$
for $B\in\mathcal E_\Delta$ and $f$ measurable. In the Markov process literature, functions $f$ on $E$ are commonly extended to $E_\Delta$ by setting $f(\Delta)=0$, and this conveniently allows us to replace $\zeta$ in the definition of $U$ by $\infty$. In what follows we will study how finiteness of $I^f_\infty$ (almost surely and with positive probability) is related to the law of the potential operator $Uf$. 

Several complementary concepts of transience exist for Markov processes. The classical definition is that a standard Markov process on state space $(E,\mathcal E)$ is transient if there exists a strictly positive, universally measurable function $h$ such that $Uh$ is finite everywhere. This concept is not strong enough for our purposes, and we instead define a \emph{transient standard Markov process} as one which simultaneously satisfies the following two conditions:
\begin{enumerate}
	\item $U(\,\cdot\, ,K)$ is bounded for all compact $K$;
	\item $\P_z(L_K < \infty) = 1$ for all compact $K$, all $z\in E$.
\end{enumerate}
Chung and Walsh \cite{ChungWalsh} §3.7 show that both conditions above imply classical transience, and in Corollary 2.3 of his seminal paper Getoor \cite{Getoor80} demonstrated that, under a regularity condition on the excessive functions of $X$, %called (LSC), 
they are equivalent to it. If further
\begin{enumerate}[resume]
	\item $\P_z(L_K < \zeta) = 1$ for all compact $K$, all $z\in E$,
\end{enumerate}
we say that $X$ is \emph{strongly transient}. Clearly any transient process with an almost surely infinite lifetime is also strongly transient. It is worth mentioning that for Lévy processes all the definitions of transience discussed above are automatically equivalent - see for example Sato \cite{Sato} Theorem 35.4 - and in addition are equivalent to the condition that
\begin{enumerate}[resume]
	\item $\displaystyle \lim_{t\to\infty} \abs{X_t}=\infty$ almost surely.
\end{enumerate}
Our analysis of path integrals uses the trick of viewing a path integral as the explosion time of a time-changed process. To define the time-change we use the finite-time path integrals
\begin{align*}
	I^f_t = \int_0^t f(X_s) \dd s, \quad t\in[0,\infty).
%	\qquad I_\infty^f \coloneqq  \lim_{t\to\infty} I^f_t = \int_0^\infty f(X_s) \dd s.
\end{align*} 
Note that the convention $f(\Delta)=0$ allows us to write either $t\in [0,\infty)$ or $t\in [0,\zeta)$. The right-continuous inverse of $(I^f_t$, $t\ge0)$ will be denoted by $\varphi^f$.
%\begin{align*}
%	\varphi_t^f 
%	& \coloneqq \Big(\int_0^\cdot f(X_u) \dd u\Big)^{-1} (t).
%	= \inf \Big\{s>0 : \int_0^s f(X_u) \dd u > t \Big\}, \quad t\in[0,\infty), \\
%	\varphi_\infty^f & \coloneqq \lim_{t\to\infty} \varphi^f_t.
%\end{align*}
Each $\varphi_t^f$ is a stopping time for $X$. When $f$ is unambiguous we shall drop it from the notation. 
%We will also be interested in the quantities
%\begin{align*}
%	\bar I^f \coloneqq \sup_{t\in[0,\infty) : I^f_t < \infty}  I^f_t,  
%	& &\bar\varphi^f \coloneqq \sup_{t\in[0,\infty) : \varphi^f_t < \infty}  \varphi^f_t.
%\end{align*}
%The connection between them is that
%\begin{align*}
%	\varphi^f ( \bar I^f ) = \varphi_{\infty}^f 
%	\qquad \text{and} \qquad
%	I^f(\varphi_{\infty} ) = I^f( \bar\varphi^f ) =  \bar I^f.
%\end{align*}
%The map $t\mapsto I_t$ is non-decreasing, left-continuous everywhere, continuous for $t\in[0,\varphi_\infty)$, and constant equal to $I_{\infty}^f$ for $t > \varphi^f_\infty$, though that constant may be $+\infty$. The map $s\mapsto \varphi_s$ is right-continuous everywhere, and continuity of $I$ yields that it is strictly increasing for $s\in[0,\bar I^f)$, and constant equal to $\varphi_{\infty}^f$ afterwards, though that constant might be $+\infty$. But $I$ is not necessarily strictly increasing, and therefore $\varphi$ is not continuous. All of these hold for every $w\in D$, rather than almost surely, because they come from properties of the Lebesgue measure.
The time-changed process $(Y_t$, $t\in[0,\infty])$ of interest is defined as
\begin{align}\label{time change}
	Y_t = X_{\varphi_t}\text{ for } t\in[0,\infty), \qquad Y_\infty = \Delta,
\end{align}
which moves on the same state space $E$ as $X$. It is not immediately clear which of the properties of $X$ are inherited by $X_\varphi$. Volkonskii \cite{Volkonskii58} proved that $X_\varphi$ is a strong Markov process if %$\varphi_t(w)$ is non-decreasing and right-continuous for every $w \in D$ and $\varphi_{\tau + t} - \varphi_\tau= \varphi_t \circ \theta_{\varphi_\tau}$ for all $t>0$ and all stopping times $\tau$. We have just seen that $\varphi$ is non-decreasing and right-continuous. We can also verify \eqref{volk}, but only under the condition that almost surely,
\begin{align}\label{volk2}
	\text{$t\mapsto I_t$ is continuous on the whole of $[0,\infty)$,}
\end{align}
which is equivalent to assuming $ I_{\infty}=\bar I:=\sup_{t\in[0,\infty) : I_t < \infty}  I_t.$ Note that
\begin{align}\label{leif}
	\varphi_{\bar I}=\varphi_\infty:=\lim_{t\to\infty} \varphi_t.
\end{align}
A sufficient condition for \eqref{volk2} to hold almost surely, though not necessary, is that
\begin{align}\label{volk4}
	\text{$I_t < \infty$ for all $t\in[0,\infty)$ almost surely}
\end{align}
which is what we are going to check when we use the strong Markov property of the time-change. This implies (indeed is equivalent to) the property $\varphi_{\infty} = \infty$ almost surely, and we get
% from \eqref{volk3.1} we get
\begin{align}\label{volk3.2}
	\int_0^\infty \sigma(X_{\varphi_s}) \dd s = \int_0^\infty \sigma(X_t) f(X_t) \dd t  \quad \text{almost surely.}
\end{align}
This will be of particular use to us later in this section.

As alluded to earlier, in contrast to existing theory of stable SDEs for $\alpha\in (1,2]$, our analysis is based on general Markov process theory. For this reason we now introduce an object which commonly appears in potential theory of Markov processes, the so-called fine topology. If $M$ is a Borel set then $x\in E$ is called a regular point for $M$ if $\P_x(T_M>0)=0$. We denote the set of regular points by $M^r$, and note that $X_{T_M}\in M^r\cup M$ $\P_y$-almost surely for all $y\in E$. A measurable set is called \emph{finely closed} if $M^r\subseteq M$ and finely open if the complement is finely closed. The collection of finely open sets with compact closure $\overline M$ in $E$ forms the base of a topology called the fine topology. It is immediately seen that right-continuity of paths of $X$ implies that all open sets are finely open, and therefore that closed sets are finely closed.

We will also make use of a result on the regularity of $q$-excessive functions, for which we refer to Theorem VI(4.9) of Blumenthal and Getoor \cite{BG} or Theorem 13.80 of Chung and Walsh \cite{ChungWalsh}. If $X$ is a standard Markov process on $(E,\mathcal E)$ with strong Feller resolvents (and dual resolvents, which is automatically fulfilled for Lévy processes) then the following are equivalent.
\begin{enumerate}
	\item[\namedlabel{Hunt condition}{(H)}] If $K\in E$ is a non-polar set then some point of $K$ must be regular for $K$;
	\item[\namedlabel{Hunt regularity}{(R)}] If $f$ is a locally integrable $q$-excessive function for some $q>0$, then $t\mapsto f(X_t)$ is continuous whenever $t\mapsto X_t$ is continuous on $[0,\zeta)$.
	\item[\namedlabel{Hunt regularity alt}{(R$)^\prime$}] If $f$ is a locally integrable $q$-excessive function for some $q>0$, and $T_n$ is an increasing sequence of stopping times with limit $T$, then $f(X_{T_n})\to f(X_T)$ on $\{T < \zeta\}$.
\end{enumerate}
The assumption that $X$ has strong Feller resolvents is fairly strong for general standard Markov $X$, but Hawkes \cite{Hawkes} showed that for Lévy processes it is equivalent to existence of a density $u^q$ for the $q$-potential measure, which is known to hold for all symmetric stable processes. Property \ref{Hunt condition} is commonly known as Hunt's condition, see for instance Chapter 13 of Chung and Walsh \cite{ChungWalsh} for equivalent formulations. Hunt's condition is satisfied by symmetric Lévy processes, but for general Lévy processes the question of whether or not it holds remains an open problem (see for instance Chapter VI.4 in Blumenthal and Getoor \cite{BG} and Chapter II.7 in Bertoin \cite{Bertoin}). Property \ref{Hunt regularity} is called regularity of $f$, and it is interesting to compare it with the fact that for every standard Markov process $X$, $t\mapsto f(X_t)$ is right-continuous and has left-hand limits on $[0,\infty)$ for any $q$-excessive function $f$.
\subsection{Positive Probability Case}\label{mainthm}

The classification of finiteness and the computation of distributions of path integrals $I^f_\infty$ and $I^f_t$ has a rich history. Integral tests were derived for different classes of stochastic processes under different assumptions on $f$. Applications can be found for instance in SDEs, branching processes, or random walks in random environments. The closest article in its general form was recently published by Kolb and Savov \cite{KS}. In the case of a L\'evy process drifting to $+\infty$ a complete integral-test characterisation of finiteness was proved under rather strong assumptions on $f$ - most importantly, $f$ was assumed finite:
\begin{theorem}[Kolb/Savov \cite{KS}]
If the measurable function $f:\R\to [0,\infty)$ is either continuous or ultimately non-increasing and $X$ is a transient Lévy process with limit $+\infty$, the following are equivalent:
\begin{enumerate}
	\item $\P_z\big( \int_0^\infty f(X_s) \dd s = \infty\big) > 0;$
	\item $\P_z\big( \int_0^\infty f(X_s) \dd s = \infty\big) = 1;$
	\item The integral test $\int_{\R\setminus B} f(x) \, U(z, \ddd x) = \infty$ holds for all Borel sets $B$ such that $\P_z(L_B < \infty) = 1$. 
\end{enumerate}
\end{theorem}
The sets appearing in (ii) are called \emph{$\P_z$-transient} sets. The occurrence of transient sets is clearly necessary as the finite time occupation of a transient set does not influence finiteness of the path integral due to the assumptions on $f$.
%The set $B$ in the integral test above has the property that after some finite time, the process leaves it and never returns. This property is usually called \emph{$\P_z$-transience} of the set, because of its similarity to transience as a property of the process, which we saw in the previous section.
In Kolb and Savov's theorem either of the two assumptions on $f$, in combination with the fact that $X$ is a Lévy process, is enough to ensure that the zero-one law $\P_z(I^f_\infty<\infty)\in \{0,1\}$ holds. It might be unclear how to use the theorem due to the occurrence of the stochastically defined transient sets. Kolb and Savov provide several examples of L\'evy processes for which the sets can be characterised. For example, if $X$ has local time at points, the theorem works with $B=(-\infty,a]$, $a\in\R$, and recovers the integral test at infinity of \cite{DK1}. 

%The Kolb-Savov proof is based on the basic estimate 
%\begin{align*}
%	\beta_{z,f} \E_z[I^f_\infty]\leq \P_z(I^f_\infty<\infty),\quad \beta_{z,f}\geq 0,
%\end{align*}
%and arguments (using enough regularity of $f$ and the L\'evy property) that ensure $\beta_{z,f \1_{\R\backslash B}}>0$ if the probability is $1$.
The present article extends the Kolb-Savov theorem in two directions. First, we shall allow extended measurable functions $f:\R\to [0,+\infty]$ without any further assumptions, and second, we shall consider $X$ to be a general standard Markov process. The generalisation to extended measurable functions is crucial in order to study SDEs with singular coefficients, such as $\sigma(x)=|x|^{\beta}$, since path integrals for $f=\sigma^{-\alpha}$ need to be analysed as part of the representation of solutions via a time-change. The generalisation to standard Markov processes is useful for instance to understand finite time-horizon path integrals by applying the infinite time-horizon theorem to killed processes.

The main difficulty arising from our weaker assumptions on $f$ is a breakdown of the zero-one law for path integrals, even in the case of a Lévy process. As an example, let $X$ be a symmetric $\alpha$-stable process with $\alpha\in (0,1)$. Then it is well-known that $\P_0(T_{[1,2]}<\infty)\in (0,1)$, which implies that $\P_0\left(\int_0^\infty +\infty \1_{[1,2]}(X_s)\dd s<\infty\right)\in (0,1)$. Similarly to transient sets in the setting of Kolb and Savov, avoidable sets must play a special role in the setting of path integrals for which zero-one laws fail. As a consequence we will prove different theorems for the positive probability case and the probability one case.

In this section we provide a complete classification of finiteness with positive probability of infinite-time horizon path integrals for standard Markov processes. The almost sure case and the general form of the Kolb/Savov theorem will be treated in the next section.

\begin{definition}\label{supportivedef}
Let $X$ be a standard Markov process on $E$ with lifetime $\zeta$. A Borel set $B\in \mathcal E$ is called $\P_z$-avoidable if $\P_z(D_B < \zeta) < 1$. If $B$ is avoidable then its complement $M=E\setminus B$ is called $\P_z$-supportive, and satisfies
\begin{align*}
	\P_z(X_t\in M\text{ for all }t\in [0,\zeta))>0.
\end{align*}
\end{definition}	
In general it is not at all clear what form avoidable or supportive sets should take but a vast literature exists for special processes. Here are the most relevant examples for the present article:
\begin{itemize}
	\item If $X$ is a recurrent Markov process then only polar sets are avoidable.
	\item If $X$ is a symmetric stable process on $\R$ of index $\alpha \in (1,2]$ then only the empty set is avoidable.
%	\item If $X_t = B_t + at$ is the Brownian motion on $\R$ with a positive drift $a>0$ then the avoidable sets are exactly the intervals $(-\infty,x)$ for $x<0$.
	\item If $X$ is a symmetric stable process on $\R$ of index $\alpha \in (0,1)$ then any compact set not containing $z$ is $\P_z$-avoidable.
\end{itemize}
We can now formulate our general theorem on path integrals, an integral test which is not a zero-one law, and which requires minimal assumptions on $f$ and $X$.
\begin{theorem}\label{integral test}
Let $f:E \to [0,+\infty]$ be measurable and $X$ a standard Markov process with (possibly infinite) life time $\zeta$. Let $z\in E$. 
%Then the following equivalence is in place:
%\begin{align*}
%	&\quad \P_z \Big( \int_0^\zeta f(X_s)\dd s < \infty\Big)>0\\
%	&\Longleftrightarrow \quad \text{There is a } \P_z\text{-supportive set } M \text{ such that } \int_M f(x) U(z,\mathrm d x)<\infty. 
%\end{align*}
Then the following are equivalent.
\begin{enumerate}
	\item $\P_z\big( \int_0^\infty f(X_s) \dd s < \infty\big) > 0;$
	\item The integral test $\int_{E\backslash B} f(x) \, U(z, \ddd x) < \infty$ holds for a $\P_z$-avoidable set $B$.
\end{enumerate}
\end{theorem}
Just as the occurrence of transient sets is clearly needed for the Kolb/Savov theorem, the occurrence of avoidable (or supportive) sets is needed here since $f$ can be anything in avoidable sets without harming the positivity of the probability of finite path integral. As for the Kolb/Savov theorem extra knowledge on avoidable sets for particular processes (such as the Wiener test for stable processes) can turn the integral test into a purely analytic statement.

%Recalling that we call the complements of avoidable sets supportive the theorem could be formulated as the equivalence of
An equivalent formulation of Theorem \ref{integral test} using the supportive sets of Definition \ref{supportivedef} is that the following are equivalent.
\begin{enumerate}
	\item $\P_z\big( \int_0^\infty f(X_s) \dd s = \infty\big) =1;$
	\item The integral test $\int_{M} f(x) \, U(z, \ddd x) =\infty$ holds for all $\P_z$-supportive sets $M$.
\end{enumerate}
%We prefer the formulation from the theorem as the lack of a zero-one law becomes more apparent.  
%In the proofs we will often stick to the notion of supportive sets as those are constructed as the set of `safe points', 
Supportive sets will appear most often in our proofs as an explicitly constructed set of `safe points', 
in the sense that those are the issuing points of the state space from which the path integral remains finite with positive probability. The idea of the proof of Theorem \ref{integral test} is to use the observation that $I^f_\infty<\infty$ is equivalent to finite-time explosion of the time-changed strong Markov process $X_\varphi$ from \eqref{time change} if the time-change is well-behaved in the sense of \eqref{volk4}. The latter leads to the introduction of an additional indicator over a so-called super-finite set for the pair $(X,f)$, in the forthcoming Lemma \ref{super finite technical lemma 2} and Proposition \ref{super finite existence}. For the analysis of explosion we exploit an old argument from the general study of transience due to Getoor \cite{Getoor80}, see Proposition \ref{Getoor++}.

%If $X$ is recurrent then Borel sets have either potential zero or infinity, and the integral test is trivial. Most applications of Theorem \ref{integral test} in later sections will be for transient standard Markov processes, but it is worth remembering that there is no dichotomy for standard Markov processes, it is possible for them to be neither recurrent nor transient.

%------------------------------------------------------------------------------

%------------------------------------------------------------------------------

%\subsection{Super-Finite Sets and a Proof of Theorem \ref{integral test}}\label{integral test proof}

%The key to proving Theorem \ref{integral test} is to define the class of `super-finite' sets, which uniformly bound $\int_0^\infty f(X_s)\dd s$ in a particular way, and prove that super-finite sets have finite potential. We will then show that in some settings super-finite sets are supportive, and that will complete the proof. 

In all of this section $X$ is a standard Markov process on state space $(E,\mathcal E)$, and $f:E \to [0,+\infty]$ is a measurable function.

\begin{definition}
A \emph{super-finite} set for $(X,f)$ is defined in relation to some $n\in\mathbb N$ and $c\in(0,1)$ by
$$
	M = \Big\{ y \in E : \P_y\Big(\int_0^\infty f(X_s) \dd s \le n\Big) > c \Big\}.
$$
\end{definition}
The regularity condition on $X$ allows that the pullback of $(c,\infty)$ by $x\mapsto\P_x(I^f_{\infty} \le n)$ is a set in $\mathcal E_\Delta$, and it is trivial to remove the $\Delta$ point and see that $M \in \mathcal E$.

This first lemma shows that if there exists some super-finite set for $(X,f)$ then it is possible to choose another measurable $g:E \to [0,\infty]$ such that the whole state space $E$ is super-finite for $(X,g)$. The constant $c^2$ is not optimal, but it suffices for our use of this lemma.

\begin{lemma}\label{super finite technical lemma 1}
Suppose
$
	M= \{ y \in E : \P_y\big(\int_0^\infty f(X_s) \dd s \le n\big) > c\}
$ 
is a non-empty super-finite set for $(X,f)$ and let $g= \mathbf 1_M\cdot f$. Then
\begin{align*}
	\P_y\Big(\int_0^\infty g(X_s) \dd s \le 2 n\Big) > c^2, \quad \forall y\in E.
\end{align*}
That is, the entire state space $E$ is super-finite for $(X,g)$.
\end{lemma}
\begin{proof}
Since $g\le f$, it also holds that $I^g_t \le I^f_t$ for all $t\in[0,\infty]$. Then the result holds immediately for $y\in M$, since by definition of that set
\begin{align*}
	\P_y(I^g_\infty \le 2 n)
	\ge \P_y(I^f_\infty \le 2 n)
	\ge \P_y(I^f_\infty \le n)
	> c > c^2.
\end{align*}
We now prove the result for regular points $y\in M^r$.  Lemma I.10.19 of \cite{BG} gives a nested increasing sequence of compact sets $K_m \subseteq M$, $m\in\mathbb N$, such that, $\P_y$-almost surely, $T_{K_m}\downarrow T_M$ as $m\to\infty$. Since $\P_y(T_M = 0)=1$ it then follows that
\begin{align*}
	& \P_y\Big( \int_{0}^{T_{K_1}} g(X_s)\dd s \le n \Big)\\
	& =\P_y\Big(\bigcap_{m=1}^\infty \Big\{ \int_{T_{K_m}}^{T_{K_1}}g(X_s)\dd s \le n \Big\}\Big)\\
	& =\lim_{m\to\infty}\P_y\Big(  \int_{T_{K_m}}^{T_{K_1}}g(X_s)\dd s \le n \Big)\\
	& \ge \lim_{m\to\infty}\P_y\Big(  \int_{T_{K_m}}^{T_{K_1}}g(X_s)\dd s \le n;\, T_{K_m}<\infty \Big)\\
	& =\lim_{m\to\infty}\int_{K_m} \P_a\Big(\int_0^{T_{K_1}}g(X_s)\dd s \le n \Big)\P_y(X_{T_{K_m}} \in \dd a;\, T_{K_m}<\infty)\\
	& \ge\lim_{m\to\infty}\int_{K_m} \P_a\Big(  \int_0^\infty f(X_s)\dd s \le n \Big)\P_y(X_{T_{K_m}}\in \dd a;\, T_{K_m}<\infty).
\end{align*}
From Lemma I.11.4  of \cite{BG} we have that $X_{T_{K_m}}\in K_m\subseteq M$ almost surely, and thus we obtain the lower bound
$$
	\P_y\Big( \int_{0}^{T_{K_1}} g(X_s)\dd s \le n \Big) > c \cdot \lim_{m\to\infty}\P_y(T_{K_m}<\infty) = c\cdot\P_y(T_M<\infty) = c.
$$
In addition, since $T_{K_m}$ decreases to $T_M=0$ $\P_y$-almost surely, it follows that
$
	\P_y( A ;\, T_{K_m} < \infty) \uparrow \P_y(A)
$
for any $A\in \mathcal F$. This, in combination with the inequality above, gives that there exists a choice of $j\in\mathbb N$ such that 
$$
	 \P_y\Big( \int_0^{T_{K_1}} g(X_s)\dd s \le n ;\, T_{K_j} < \infty\Big) > c.
$$
Monotonicity of the sequence of sets $(K_m$, ${m\in\mathbb N})$ yields $T_{K_1} \ge T_{K_j}$, and thus
\begin{align}\label{eq3.1}
	\P_y\Big( \int_0^{T_{K_j}} g(X_s)\dd s \le n ;\, T_{K_j} < \infty\Big) > c.
\end{align}
Fix this $j$. Next, using the strong Markov property, we get
\begin{align*}
	&\P_y\Big(\int_0^\infty  g(X_s)\dd s \le 2n ;\, T_{K_j} < \infty\Big)\\
	&\ge \P_y
	\Big(
		\int_0^{T_{K_j}} g(X_s)\dd s \le n;\, 
		\int_{T_{K_j}}^\infty  g(X_s)\dd s \le n;
		\, T_{K_j} < \infty
	\Big)\\
	&= \E_y
	\Big[ 
		\1_{\big(\int_0^{T_{K_j}}  g(X_s)\dd s \le n\big)}  
		\1_{(T_{K_j} < \infty)} 
		\E_y\Big[ \1_{\big(\int_{T_{K_j}}^\infty  g(X_s)\dd s \le n\big)} \, \Big | \mathcal F_{T_{K_j}}\Big] 
	\Big]\\
	&= \E_y
	\Big[ 
		\1_{\big(\int_0^{T_{K_j}}  g(X_s)\dd s \le n\big)}  
		\1_{(T_{K_j} < \infty)} 
		\P_{X_{T_{K_j}}}\Big(  {\int_0^{\infty}  g(X_s)\dd s \le n}\Big)
	\Big],
	\intertext{and now since $I^g_{\infty} \le I^f_{\infty}$,}
	&\ge  \E_y
	\Big[ 
		\1_{\big(\int_0^{T_{K_j}}  g(X_s)\dd s \le n\big)}  
		\1_{(T_{K_j} < \infty)} 
		\P_{X_{T_{K_j}}}\Big(  {\int_0^\infty  f(X_s)\dd s \le n}\Big) 
	\Big].
\end{align*}
Since $K_j\subseteq M$ the inner probability is bounded below by $c$. This holds even in the extreme case $K_j=M$, because then $M$ is compact and $X_{T_{K_j}} \in M$ almost surely. The remaining expectation can also be bounded from below by $c$ using \eqref{eq3.1}. In total this leads to
\begin{align*}
	\P_y\Big(\int_0^\infty  g(X_s)\dd s \le 2n\Big) 
	\ge \P_y\Big(\int_0^\infty  g(X_s)\dd s \le 2n;\, T_{K_j} < \infty\Big) 
	> c^2.
\end{align*}	
Thus the lemma is proved for regular points $y\in M^r$. What now remains is to extend it to all $y\in E$. In this case,
\begin{align*}
	&\P_y\Big(\int_0^\infty  g(X_s)\dd s \le 2n\Big) \\
	&\quad = \P_y\Big(\int_0^\infty  g(X_s)\dd s \le 2n ;\, T_M < \infty \Big) + \P_y\Big(\int_0^\infty  g(X_s)\dd s \le 2n ;\, T_M = \infty \Big)\\
	&\quad = \P_y\Big(\int_{T_M}^\infty  g(X_s)\dd s \le 2n ;\, T_M < \infty \Big) + \P_y(T_M = \infty), \\
	&\quad = \int_E \P_a\Big(\int_0^\infty  g(X_s)\dd s \le 2n \Big) \P_y(X_{T_M}\in \dd a;\, T_M < \infty ) + \P_y(T_M = \infty).
\end{align*}
Lemma I.11.4  of \cite{BG} tells us that the integrating measure is concentrated on $M\cup M^r$, and thus we can use what we have already proved for elements of  $M\cup M^r$ to conclude that
\begin{align*}
	\P_y\Big(\int_0^\infty  g(X_s)\dd s \le 2n\Big) 
	& > c^2\P_y(T_M < \infty) + \P_y(T_M = \infty) \\
	& \ge c^2\big(\P_y(T_M < \infty) + \P_y(T_M = \infty)\big) = c^2.
\end{align*}
The proof of the lemma is now complete.
\end{proof}

The next lemma serves as preparation for the proposition that follows it. It shows that restricting $f$ to a super-finite set ensures that the path integral does not explode in finite time, which is exactly condition \eqref{volk4}, and ensures that the time-change of $X$ with the inverse of the path integral is strong Markov.

\begin{lemma}\label{super finite technical lemma 2}
Suppose
$
	M= \big\{ y \in E : \P_y\big(\int_0^\infty f(X_s)\dd s \le n\big) > c\big\}
$
is a non-empty super-finite set for $(X,f)$. Let $g= \mathbf 1_M\cdot f$. Then
 \begin{align*}
 	\P_y\Big(\exists\, t\in(0,\infty): \int_0^t g(X_s)\dd s =\infty\Big)=0,\quad \forall y\in E.
 \end{align*}
\end{lemma}
\begin{proof}
Fix $y\in E$ and recall the stopping times
$$
	\varphi^g_n = \inf\{t>0 :  I^g_t > n\},\; n\in\mathbb N, \qquad \varphi^g_{\infty} = \lim_{n\to\infty} \varphi^g_n.
$$
Since
\begin{align}\label{eq3.g}
	\varphi^g_n<\infty\quad \Leftrightarrow\quad \int_0^\infty g(X_s)\dd s >n\quad \Leftrightarrow\quad \exists\, t\in(0,\infty): \int_0^t g(X_s)\dd s>n,
%	\varphi_n<\infty\quad \Leftrightarrow\quad I^g_\infty >n\quad \Leftrightarrow\quad \exists t>0: I^g_t >n,
\end{align}
we obtain from Lemma \ref{super finite technical lemma 1} that for all $y\in E$,
\begin{align}\label{eq3.10}
	\P_y(\varphi^g_{2n}<\infty)\le 1-c^2.
\end{align}
We then see for $k,n\in\mathbb N$ that
\begin{align*}
	\P_y(\varphi^g_{2kn}<\infty)
	& = \P_y\Big(\exists t>0 : \int_0^t g(X_s)\dd s > 2kn\Big) \\
	& = \P_y\Big(\varphi^g_{2n(k-1)} < \infty\,;\,\exists t'>0 : \int_{\varphi^g_{2n(k-1)}}^{\varphi^g_{2n(k-1)} + t'} g(X_s)\dd s > 2n\Big) \\
	&  = \E_y
	\Big[ 
		\mathbf 1_{(\varphi^g_{2n(k-1)} < \infty)} 
		\E_y
		\Big[
			\mathbf 1_{\big(\exists t'>0 : \int_{\varphi^g_{2n(k-1)}}^{\varphi^g_{2n(k-1)} + t'} g(X_s)\dd s > 2n\big)} 
			\Big| \mathcal F_{\varphi^g_{2n(k-1)}}
		\Big]
	\Big] \\
%	&  = \E_y
%	\Big[ 
%		\mathbf 1_{(\varphi^g_{2n(k-1)} < \infty)} 
%		\P_y
%		\Big(
%			\exists t'>0 : \int_{\varphi^g_{2n(k-1)}}^{\varphi^g_{2n(k-1)} + t'} g(X_s)\dd s > 2n
%			\Big| \mathcal F_{\varphi^g_{2n(k-1)}}
%		\Big)
%	\Big] \\
	& = \E_y
	\Big[  
		\mathbf 1_{(\varphi^g_{2n(k-1)} < \infty)}
		\P_{X_{\varphi^g_{2n(k-1)}}}\Big(\exists t'>0 : \int_0^{t'} g(X_s)\dd s > 2n \Big) 
	\Big] \\
	& = \E_y\Big[\mathbf 1_{(\varphi^g_{2n(k-1)} < \infty)} \P_{X_{\varphi^g_{2n(k-1)}}}\Big(\varphi^g_{2n} < \infty \Big)   \Big],
\intertext{and by \eqref{eq3.10},}
	& \le (1-c^2) \P_y(\varphi^g_{2n(k-1)} < \infty) \\
	&\qquad\quad \vdots \\
	& \le (1-c^2)^k\rightarrow 0, \quad \text{as } k\to\infty.
\end{align*}
Continuity of measures yields $\P_y(\varphi^g_{\infty} < \infty) = 0$, which via \eqref{eq3.g} implies the claim.
\end{proof}

The next proposition is partially motivated by ideas from Getoor \cite{Getoor80}, in particular the proof of Lemma (3.1) there, albeit used in a different fashion. It proves that super-finite sets for $(X,f)$ have finite $f$-potential.

\begin{proposition}\label{Getoor++}
Suppose for $n\in\mathbb N$ and $p\in(0,1)$ that 
$
	M= \big\{ y \in E : \P_y\big(\int_0^\infty f(X_s) \dd s \le \frac n 2\big) > p\big\}
$
is a non-empty super-finite set for $(X,f)$. Then
$$
	\int_M f(x) U(y,\mathrm d x)\leq \frac n {p^2} \quad \text{ for all }y\in E.
$$
\end{proposition}
\begin{proof}
With $g=\mathbf 1_M\cdot f$ we introduce the time-changed process $Y_t = X_{\varphi_t^g}$ from \eqref{time change}, where
$$
	\varphi^g_t = \inf \Big\{ s>0 : \int_0^s g(X_u)\dd u > t\Big\}, \quad t\in[0,\infty).
$$
Lemma \ref{super finite technical lemma 2} tells us that \eqref{volk4} and thus \eqref{volk2} are fulfilled almost surely and that therefore, as per the discussion at the start of the chapter concerning the work of Volkonskii \cite{Volkonskii58}, $Y$ is a strong Markov process. For this proof we will use the notation
\begin{align}\label{h}
	h(x) = \P_x(I^g_{\infty} \le n),\quad x \in E.
\end{align}
By Lemma \ref{super finite technical lemma 1} $h$ is bounded below by $p^2$ on $E$. 

%$Y$ is a strong Markov process on the same state space as $X$, but which only takes values in $M$\footnote{or also closure?}. 
We denote by $(P_t)$ the transition operator of $Y$ and by $U_Y$ the corresponding potential operator, to distinguish it from $U_X$ the potential operator of $X$. The function $h$ from \eqref{h} is extended to $E_\Delta$ by setting $h(\Delta)=0$ as usual. We then see that
\begin{align*}
	P_n \1_E(x)
	& = \P_x(X_{\varphi_n^g}\in E) \\
	& = \P_x(\varphi_n^g < \zeta_X) \\
	& = \P_x\Big( \exists\, s<\zeta_X \text{ such that } \int_0^s g(X_u) \dd u > n\Big) \\
	& \le \P_x\Big( \int_0^{\zeta_X} g(X_u) \dd u > n\Big) \\
	& = \mathbf 1_E(x)- h(x)
\end{align*}
for all $x\in E_\Delta$. The main part of the proof is showing that $U_Y h$ is bounded above. Using the definition of the potential, the estimate of $P_n \1_E$ above, and the semigroup property yields
\begin{align}\label{s}
% Don't use split here because it formats the document oddly, in order to fit the whole (long) thing on one page.
	U_Y h (x)
	&= \lim_{t\to\infty}\int_0^t P_s h(x)\dd s\nonumber\\
	&\leq \lim_{t\to\infty} \int_0^t  \big( P_s \mathbf{1}_E(x)- P_s  P_n\mathbf{1}_E(x)\big)\dd s\nonumber\\
	&= \lim_{t\to\infty} \Big(\int_0^t   P_s \mathbf{1}_E(x)\dd s -\int_0^t  P_{s+n} \mathbf{1}_E(x)\dd s\Big)\nonumber\\
	&= \lim_{t\to\infty} \Big(\int_0^t   P_s \mathbf{1}_E(x)\dd s -\int_n^{t+n}  P_s \mathbf{1}_E (x)\dd s\Big)\\
	&= \lim_{t\to\infty} \Big(\int_0^n   P_s \mathbf{1}_E(x)\dd s -\int_t^{t+n}  P_s \mathbf{1}_E (x)\dd s\Big)\nonumber\\
	&\leq \int_0^n   P_s \mathbf{1}_E(x)\dd s\nonumber\\
	&\leq n,\nonumber
\end{align}
for all $x\in E_\Delta$. It has already been noted that $h$ is bounded below by $p^2>0$ on $E$. Then by monotonicity of the potential operator,
\begin{align}\label{ss}
%	U_Y \mathbf 1_M(y) \le U_Y \mathbf 1_{A_p}(y) \le \frac 1 p  U_Y h(y)\quad \text{ for all }y\in E.
	U_Y \mathbf 1_E (y) \le \frac 1 {p^2}  U_Y h(y)\quad \text{ for all }y\in E.
\end{align}
Combining \eqref{s} and \eqref{ss} implies $U_Y\mathbf 1_E(y)\leq \frac 1 {p^2} U_Yh(y) \leq \frac{n}{p^2}$ for all $y\in E$. Bearing in mind the following identity for bounded measurable functions $b$, obtained via change of variables and the definition $g=\mathbf 1_M\cdot f$:
\begin{align*}
	U_Y b(y)
	=\E_y\Big[\int_0^{\infty}   b(X_{\varphi^g_s})\dd s\Big]
%	&=\E_y\Big[\int_0^{\varphi^-_{\infty}} \mathbf  \sigma(X_t) g(X_t)\dd t\Big]\\
	=\E_y\Big[\int_0^\infty b(X_t) g(X_t)\dd t\Big]
	=\int_M  b(x) f(x) U_X(y,\mathrm d x),
\end{align*}
we obtain the desired estimate
$$
	\int_M f(x) U_X(y,\mathrm d x) = U_Y\mathbf 1_E(y) < \frac n {p^2}\quad \text{ for all }y\in E.
$$
The proof is now complete.
\end{proof}

Proposition \ref{Getoor++} is a strong result for super-finite sets. All that remains to be done in order to prove Theorem \ref{integral test} is to prove that non-empty super-finite sets are also supportive. Before doing so, another technical lemma is needed.

\begin{lemma}\label{regular points}
For $n\in \mathbb N, c\in (0,1)$, the set
\begin{align*}
	B = \Big\{x\in E :  \P_x\Big(\int_0^\infty f(X_s) \dd s \le n\Big)\le c\Big\}
\end{align*}
contains its regular points, and is therefore finely closed.
\end{lemma}
\begin{proof}
The definition of $B$ is exactly that of the complement in $E$ of some super-finite set for $(X,f)$. Let $K\subseteq B$ be a compact set, so that, by right-continuity of paths, $X_{T_K}\in K$ almost surely. Then for any $a\in E$ we obtain
\begin{align}\label{eqa}
%\begin{split}
%	&\quad \P_a \Big(\int_0^\infty f(X_s)\dd s \le n\Big)\\
%	& = \P_a\Big(\int_0^\infty f(X_s)\dd s \le n ;\, T_{K}<\zeta \Big) 
%	+ \P_a\Big(\int_0^\infty f(X_s)\dd s \le n ;\, T_{K}\ge\zeta \Big) \\
%	& \le  \P_a\Big(\int_{T_k}^\infty f(X_s)\dd s \le n ;\, T_{K}<\zeta \Big) 
%	+ \P_a\Big(\int_0^\infty f(X_s)\dd s \le n ;\, T_{K}\ge\zeta \Big) \\
%	& =  \int \P_y\Big(\int_{0}^\infty f(X_s)\dd s \le n\Big) \P_a(X_{T_K}\in \dd y;\, T_{K}<\zeta ) 
%	+ \P_a\Big(\int_0^\infty f(X_s)\dd s \le n ;\, T_{K}\ge\zeta \Big) \\
%	& \le c+ \P_a(T_{K}\ge\zeta),
%\end{split}\\
\begin{split}
	\P_a \big( I_{\infty}^f \le n\big)
	& = \P_a\big(I_{\infty}^f \le n ;\, T_{K}<\zeta \big) 
	+ \P_a\big(I_{\infty}^f \le n ;\, T_{K}\ge\zeta \big) \\
	& \le  \P_a\Big(\int_{T_K}^\infty f(X_s)\dd s \le n ;\, T_{K}<\zeta \Big) 
	+ \P_a\big(I_{\infty}^f \le n ;\, T_{K}\ge\zeta \big) \\
	& =  \int \P_y\big(I_{\infty}^f \le n\big) \P_a(X_{T_K}\in \dd y;\, T_{K}<\zeta ) 
	+ \P_a\big(I_{\infty}^f\le n ;\, T_{K}\ge\zeta \big) \\
	& \le c+ \P_a(T_{K}\ge\zeta),
\end{split}
\end{align}
as $K\subseteq B$. Since \eqref{eqa} is true for all compact sets $K\subseteq B$, it then holds for all $a\in E$ that
\begin{align}\label{eqaa}
	\P_a\Big(\int_0^\infty f(X_s)\dd s \le n\Big)\le c + \inf_{K\subseteq B}\P_a(T_{K}\ge\zeta).
\end{align}
We now argue that the second summand of the right-hand side of \eqref{eqaa} equals $0$ if $a$ is regular for $B$. From Lemma I.10.19 of \cite{BG} it follows that there exists an increasing sequence $(K_n)$ of compact subsets of $B$ such that $\P_a(T_{K_n}\downarrow T_B)=1$. Monotonicity of measures implies $\inf_{K\subseteq B}\P_a(T_{K}\ge\zeta)=\lim_{n\to\infty} \P_a(T_{K_n} \ge \zeta) = \P_a(T_B \ge \zeta)=0$.
%The fact that the stopping times $T_{K_n}$ are decreasing as $n\to\infty$ yields
%\begin{align*}
%	\inf_{n\in\mathbb N} \P_a(T_{K_n}\ge\zeta) 
%	& \le \limsup_{n\to\infty} \E_a\big[\1_{(T_{K_n}\ge\zeta)}\big],
%\intertext{and hence, we can apply the reverse Fatou lemma to see that for the regular points $a\in B^r$,}
%	\inf_{n\in\mathbb N} \P_a(T_{K_n}\ge\zeta)
%	& \le  \E_a\big[\limsup_{n\to\infty} \1_{(T_{K_n}\ge\zeta)}\big] \\
%	& = \P_a\big(\limsup_{n\to\infty} T_{K_n} \ge \zeta\big) \\
%	& = \P_a(T_B \ge \zeta) = 0
%\end{align*}
But then \eqref{eqaa} implies $a\in B$. Thus $B^r\subseteq B$. Equivalence of $B$ being finely closed and containing its regular points is Exercise II.4.9 of \cite{BG}.
\end{proof}

We have built enough structure around super-finite sets to now prove the final proposition of this section, which establishes that super-finite sets are supportive when the process is issued from them.

\begin{proposition}\label{super finite existence}
For $n\in\mathbb N$, $c\in(0,1)$ the super-finite set 
$$
	M_{n,c}=\Big\{y\in E : \P_y\Big(\int_0^\infty f(X_s) \dd s \le n\Big)> c\Big\}
$$
is $\P_z$-supportive if and only if $z\in M_{n,c}$. In particular, if
$
	\P_z(I_\infty^f < \infty)>0
$ 
holds for some fixed $z\in E$ then there exists a super-finite set for $(X,f)$ which is also $\P_z$-supportive.
\end{proposition}
\begin{proof}
%The argument relies crucially on the measurable functions $h_n$ from \eqref{h} and the measurable sets 
For ease of notation let
$
	h_n(x) = \P_x\big( \int_0^\infty f(X_s)\dd s \le n\big)
$
and define 
$
	B_{n,c} = \{x\in E : h_n(x)\le c\}= E \setminus M_{n,c}
$	
as in Lemma \ref{regular points}. 

Necessity of $z\in M_{n,c}$ for $M_{n,c}$ to be supportive is clear, and we shall prove sufficiency by proving the contrapositive. Let us suppose for the moment that $M_{n,c}$ is not $\P_z$-supportive, that is,
$$
	\P_z(D_{B_{n,c}}<\zeta)=\P_z(D_{E \setminus M_{n,c}}<\zeta)=1.
$$ 
From this we obtain that
\begin{align}\label{eq11a}
\begin{split}
	h_n(z) %&= \P_z\Big(\int_0^\infty f(X_s)\dd s \le n\Big)\\
	&= \P_z\Big(\int_0^\infty f(X_s)\dd s \le n;\, D_{B_{n,c}}<\zeta\Big) \\
	&\le \P_z\Big(\int_{D_{B_{n,c}}}^\infty f(X_s)\dd s \le n;\, D_{B_{n,c}}<\zeta\Big) \\
	&=\int_E \P_a\Big(\int_0^\infty f(X_s)\dd s \le n\Big) \P_z(X_{D_{B_{n,c}}}\in\dd a ;\,D_{B_{n,c}}<\zeta).
\end{split}
\end{align}
Due to Lemma \ref{regular points} the regular points of $B_{n,c}$ belong to $B_{n,c}$, and therefore Lemma I.11.4 of \cite{BG} tells us that $\P_y(X_{D_{B_{n,c}}}\in\dd a ;\, D_{B_{n,c}}<\zeta)$ is concentrated on $B_{n,c}$. Then returning to \eqref{eq11a} we can deduce from the definition of $B_{n,c}$ that
\begin{align*}
	h_n(z)
	&\le \int_E \P_a\Big(\int_0^\infty f(X_s)\dd s \le n\Big) \P_z(X_{D_{B_{n,c}}}\in\dd a ;\,D_{B_{n,c}}<\zeta)\le c.
\end{align*}
To recap, we have proved that
$$
	M_{n,c}\text{ is not }\P_z\text{-supportive}
	\quad \Rightarrow\quad 
	h_n(z)\le c,
$$
or, equivalently,
\begin{align}\label{L}
	h_n(z)> c
	\quad \Rightarrow \quad 
	M_{n,c} \text{ is }\P_z\text{-supportive}.
\end{align}
Since $z\in M_{n,c}$ if and only if $h_n(z)>c$, the first claim is proved. For the second claim, note that, by continuity of measures, $\P_z(\int_0^\infty f(X_s)\dd s < \infty)>0$ implies 
$$
	h_{n_0}(z)=\P_z\Big(\int_0^\infty f(X_s)\dd s \le n_0\Big)>0
$$ 
for some $n_0\in\mathbb N$. Hence in this case \eqref{L} implies the existence of some $c_0$ such that $M_{n_0,c_0}$ is $\mathbb P_z$-supportive.
\end{proof}

\textbf{Proof of Theorem \ref{integral test}.}
Path integral proofs typically have a simple and a complicated direction. The simple direction ``$\Leftarrow$" uses the potential expression to deduce finiteness of the expectation. In our setting the argument goes as follows. Suppose that $B$ is a $\P_z$-avoidable set with complement $M = E\setminus B$ satisfying
$$
	\E_z\Big[\int_0^\infty \mathbf 1_M (X_s) f(X_s)\dd s\Big] = \int_M f(x) U(z,\mathrm d x)<\infty.
$$
Then $\P_z ( \int_0^\infty \mathbf 1_M(X_s) f(X_s)\dd s < \infty)=1$, and thus
\begin{align*}
	\P_z \Big( \int_0^\infty f(X_s)\dd s < \infty\Big)
	&\ge \P_z \Big( \int_0^\infty f(X_s)\dd s < \infty ; X_s\in M \,\forall s<\zeta \Big) \\
	&= \P_z \Big( \int_0^\infty \mathbf 1_M(X_s) f(X_s)\dd s < \infty ; X_s\in M \, \forall s<\zeta \Big) \\
	& = \P_z( X_s\in M \text{ for all } s<\zeta) >0.
\end{align*}
This shows the ``$\Leftarrow$" direction of  Theorem \ref{integral test}.

To prove the ``$\Rightarrow$" direction we take the super-finite supportive set from Proposition \ref{super finite existence} and obtain the integral test from Proposition \ref{Getoor++}. 
\qed

%------------------------------------------------------------------------------

\subsection{Almost Sure Case}

Theorem \ref{integral test} was formulated and proved in great generality. In the following sections we derive a couple of corollaries by adding or changing the assumptions. A first variant of our proof technique provides a complete theorem in the probability one setting. The unpleasant supportive sets do not vanish in the integral test $\int_M f(x)U(z,\ddd x)<\infty$, but can be made large. 
\begin{theorem}\label{ASIT0}
Let $X$ be a standard Markov process on state space $E$ and $f:E\to[0,\infty]$ measurable. Then the following are equivalent.
\begin{enumerate}
	\item $ \P_z\big(\int_0^\infty f(X_s) \dd s < \infty\big)=1;$
	\item For every $\eps>0$ there exists a $\P_z$-supportive set $M$ such that
	$
		\int_M f(x) \,U(z,\ddd x) < \infty
	$
	and $X$ stays in $M$ with probability at least $1-\varepsilon$.
\end{enumerate}
\end{theorem}
Before going into the proof let us recall again that the theorem can be formulated equivalently using the more familiar notion of avoidable sets, the complements of supportive sets. We stick to the notion of supportive sets as those are constructed in the proof.
\begin{proof}\hfill\\
\textbf{(i)$\Rightarrow $(ii)}:
Assuming (i) it follows that for every $\eps>0$, and in particular any choice of $\eps\in(0,1/2)$, there exists an $n\in\mathbb N$ such that 
\begin{align}\label{eq3.n}
	\P_z\Big(\int_0^\infty f(X_s) \dd s \le n \Big) > 1-\eps.
\end{align}
Fix these $n,\eps$ and define the super-finite set 
$
	M=\big\{y\in E : \P_y\big(\int_0^\infty f(X_s) \dd s \le n\big)> \eps\big\},
$
which by \eqref{eq3.n} contains $z$, and thus by Propsition \ref{super finite existence} is $\P_z$-supportive. Moreover,
\begin{align*}
	&\P_z\Big(\int_0^\infty f(X_s) \dd s \le n\Big) \\
	&\qquad = \P_z\Big(\int_0^\infty f(X_s) \dd s \le n;\; D_{E\setminus M} = \infty\Big) 
	+ \P_z\Big(\int_0^\infty f(X_s) \dd s \le n;\; D_{E\setminus M} < \infty\Big) \\
	&\qquad \le \P_z(D_{E\setminus M} = \infty) + \P_z\Big(\int_{D_{E\setminus M}}^\infty f(X_s) \dd s \le n;\; D_{E\setminus M} < \infty\Big).
	\end{align*}
Applying the strong Markov property at $D_{E\setminus M} $ and recalling from Lemma \ref{regular points} that $E\setminus M$ contains its regular points, which implies that $X_{D_{E\setminus M}}\in E\setminus M$ almost surely, gives the upper bound $\P_z(D_{E\setminus M} = \infty) + \eps$. In combination with \eqref{eq3.n} this yields
$
	1-2\eps<\P_z(D_{E\setminus M} = \infty).
$
It remains only to note via Proposition \ref{Getoor++} that $\int_M f(x) \,U(z,\ddd x) < \infty$.

\textbf{(i)$\Leftarrow $(ii)}: 
Take $\eps>0$ small and let $M$ be such that $\P_z(D_{E\setminus M} < \infty)\le\eps$ and
$$
	\E_z\Big[ \int_0^\infty f(X_s)\1_M(X_s) \dd s \Big]=\int_M f(x) \,U(z,\ddd x) < \infty.
$$
Then $\P_z\big( \int_0^\infty f(X_s)\1_M(X_s) \dd s < \infty \big) = 1$ and, in particular,
\begin{align*}
	\P_z\Big(\int_0^\infty f(X_s) \dd s < \infty\Big)
	& \ge \P_z\Big(\int_0^\infty f(X_s) \dd s < \infty;\; D_{E\setminus M} =\infty\Big) \\
	& = \P_z\Big(\int_0^\infty f(X_s) \1_M(X_s)\dd s < \infty;\; D_{E\setminus M} =\infty\Big) \\
	& =  \P_z(D_{E\setminus M} =\infty)  > 1-\eps.
\end{align*}
This holds for arbitrarily small $\eps>0$, and thus $\P_z(\int_0^\infty f(X_s) \dd s < \infty)=1$.
\end{proof}

Example \ref{example} below shows what the large supportive sets might be, and why they cannot be omitted. In that example $f$ has a pole, and the large supportive sets can be chosen to be all sets that do not contain a small interval around the pole of $f$.

We can immediately derive an interesting consequence of Theorem \ref{ASIT0} in the case that $f$ is locally bounded and $X$ is transient. The main point here is that the local boundedness of $f$ implies that finiteness of path integrals is not affected by the behaviour of $X$ for a finite amount of time. Hence, speaking about sets in which $X$ spends no time (avoidable) or finite time (transient) should be the same. The first result makes this idea precise, and the Theorem following specialises to processes with trivial tail-$\sigma$-algebra (such as L\'evy processes) for which transience of sets are zero-one events. 
\begin{corollary}\label{ASIT0+}
Let $X$ be a transient standard Markov process on $\R^d$, and let $f:\R^d\to[0,\infty)$ be measurable and bounded on compact sets. Then the following are equivalent:
\begin{enumerate}
	\item $\P_z\big(\int_0^\infty f(X_s) \dd s < \infty\big)=1;$
	\item For every $\eps>0$ there exists a set $B\in\mathcal B(\R^d)$ such that
	$
		\int_{\R^d\setminus B} f(x) \,U(z,\ddd x) < \infty
	$
	and $\P_z(L_B < \zeta) \ge 1 - \eps$.
\end{enumerate}
\end{corollary}
\begin{proof}
For any Borel set $B$ it clearly holds that the last exit time $L_B$ is zero on the event that the first entry time $D_B$ is infinite. Therefore, if we assume (i), Theorem \ref{ASIT0} gives a set $B\coloneqq \R^d\setminus M$ such that $\P_z( L_B < \zeta)\ge 1-\eps$ and $\int_{\R^d\setminus B} f(x) \,U(z,\ddd x) < \infty$. So (ii) is proven.

Now suppose (ii) holds. Proposition I(9.3) of Blumenthal and Getoor \cite{BG} gives that, on the event $\{t < \zeta\}$, the random set $\{X_s : s\in[0,t]\}$, built from the points hit by $X$ up to time $t$, is almost surely bounded. It follows that for a random time $T\in[0,\infty)$ with $T < \zeta$ almost surely, the path of $X$ up to time $T$ is almost surely contained in a (proper) compact subset of $\R^d$. Therefore, if $(K_n)$ is an increasing sequence of compact sets with limit $\R^d$, it holds for fixed $\eps>0$ that $\P_z(\lim_{n\to\infty} T_{\R^d\setminus K_n}\ge (\zeta-\eps)^+) = 1$. Thus continuity of measure yields $\lim_{n\to\infty} T_{\R^d\setminus K_n}\ge \zeta$ $\P_z\text{-almost surely}$. The above argument only works because we have specified that the state space is unbounded. Fix some $\eps>0$, and take $B$ the set satisfying $\P_z(L_B < \zeta) \ge 1 - \eps$ and $\int_{\R^d\setminus B} f(x) \,U(z,\ddd x) < \infty.$ Due to continuity of measure,
$$
	\P_z(L_B \ge T_{\R^d\setminus K_n}) \downarrow \P_z(L_B \ge \zeta) < \eps.
$$
Then we can find a compact set $K\subseteq \R^d$ such that
$
	\P_z(L_B < T_{\R^d\setminus K}) > 1-2\eps.
$
Define the set $M = K \cup (\R^d\setminus B)$. The definition of $M$ implies that on the event $L_B < T_{\R^d\setminus K}$, the process $X$ never leaves $M$. That is,
$$
	\P_z( D_{\R^d\setminus M} = \infty) \ge \P_z( L_B < T_{\R^d\setminus K}) > 1-2\eps
$$
so $X$ stays in $M$ with probability at least $1-2\eps$. In addition
$$
	\int_M f(x) \,U(z,\ddd x) \le \int_{\R^d\setminus B} f(x) \,U(z,\ddd x) + U(z,K)\sup_{x\in K}f(x).
$$
This together with the integral test of (ii), boundedness of $f$ on compact sets, and $U(z,K)<\infty$ implies that
$
	\int_M f(x) \,U(z,\ddd x) < \infty.
$
Our choice of $\eps>0$ was arbitrary, and so the result follows from Theorem \ref{ASIT0}.
\end{proof}

In the case of a trivial tail $\sigma$-Algebra, Corollary \ref{ASIT0+} simplifies and for the special case of a L\'evy process matches exactly the main theorem of Kolb and Savov \cite{KolbSavov}, which we recalled above. For an unkilled Markov process $X$, we say $X$ has a trivial tail $\sigma$-algebra when issued from $z\in E$ if
$$
	A \in \bigcap_{s\ge 0} \sigma\big(X_t,\, t\ge s\big) \quad \Rightarrow \quad \P_z(A) \in\{0,1\}.
$$
As an example, a Lévy process on $\R^d$ has a trivial tail $\sigma$-algebra when issued from every $z\in\R^d$. Since the events $\{L_B < \infty\}$ are in the tail $\sigma$-algebra, the previous theorem can be reformulated using transient sets, i.e. the sets with $\P_z(L_B<\infty)=1$.

\begin{theorem}[Generalised Kolb-Savov Zero-One Law for Perpetual Integrals]\label{ASIT}
Let $X$ be a standard Markov process on $\R^d$ with trivial tail $\sigma$-algebra when issued from $z\in\R^d$, and let $f:\R^d\to [0,\infty)$ be measurable and bounded on compact sets. Suppose in addition that $\P_z(\zeta = \infty)=1$. Then the following are equivalent:
\begin{enumerate}
%	\item $\P_z\big( \int_0^\infty f(X_s) \dd s = \infty\big) > 0;$\footnote{Why now $=\infty$ and not $<\infty$?}
%	\item $\P_z\big( \int_0^\infty f(X_s) \dd s = \infty\big) = 1;$
%	\item The integral test
%	$
%		\int_{\R^d\setminus B} f(x) \, U(z, \ddd x) = \infty
%	$
%	holds for all transient sets $B$, i.e. all measurable sets with $\P_z(L_B < \infty) = 1$. 
	\item $\P_z\big( \int_0^\infty f(X_s) \dd s < \infty\big) > 0;$
	\item $\P_z\big( \int_0^\infty f(X_s) \dd s < \infty\big) = 1;$
	\item There exists a transient set $B$ (that is, a Borel set with $\P_z(L_B < \infty) = 1$) such that integral test
	$
		\int_{\R^d\setminus B} f(x) \, U(z, \ddd x) < \infty
	$
	holds. 
\end{enumerate}\end{theorem}

%------------------------------------------------------------------------------
%------------------------------------------------------------------------------
%------------------------------------------------------------------------------
% SECTION 3: SOMETHING SOMETHING

\subsection{Examples for Stable L\'evy Processes}
In this section we give two examples that go beyond the drastic example $f=+\infty \1_{[1,2]}$ from Section \ref{mainthm}. Both examples are for transient symmetric stable processes on $\R$, i.e.\ $\alpha\in (0,1)$. In that setting the potential measure is absolutely continuous with (ignoring constants)
\begin{align*}
	U(z,\dd x)=|z-x|^{\alpha-1}\dd x, \qquad z,x\in\R.
\end{align*}
The first example shows that supportive sets also cannot be omitted in the integral test for finite (or bounded) $f$. Nonetheless, the supportive sets can be chosen large as proved in Theorem \ref{ASIT0}.
\begin{example}\label{example}
Assume $z=0$, fix a point $y\in\R$ not equal zero, and let $\eps>0$ be such that $\abs{y}>\eps$. Suppose that $f$ has support on $\textbf B_\eps(y) = \{x:\abs{x-y}<\eps\}$. Since the density $u(x) = \abs{x}^{\alpha-1}$ of $U(0,\ddd x)$ is bounded on the support of $f$ we find that
$$
	\int_\R f(x) \,U(0,\ddd x) < \infty
	\qquad \Longleftrightarrow \qquad
	 \int_\R f(x) \dd x < \infty.
$$
Hence, if we define such an $f$ with a non-integrable pole at $y$, for instance $f(x) = (x-y)^{-2}\1_{\textbf B_\eps(y)}(x)$, it holds that $\int_\R f(x) \,U(0,\ddd x) =\infty$. Now recall that $X$ stays a positive distance away from the pole $y$ almost surely.\footnote{This well-known fact can be proven using the density of the point of closest reach of $X$, see \cite{deep3}.}  It follows that
$$
	\xi\coloneqq \sup_{t\in[0,\infty)}f(X_t) < \infty \qquad \text{almost surely.}
$$
In addition, $\P_0(L_{\textbf B_\eps(y)} < \infty)=1$. Therefore, the infinite-time path integral is finite almost surely, that is,
$$
	\int_0^\infty f(X_s) \dd s \le \xi L_{\textbf B_\eps(y)} < \infty,
$$
but the integral test fails without a supportive set. The supportive set can be chosen large. Any set not containing a small ball around the pole of $f$ at $y$ is possible. This is exactly the phenomenon that we found in Theorem \ref{ASIT0}.
\end{example}
The next example gives a bounded function $f$ for which the supportive set in the integral test can be given explicitly.
\begin{example}\label{ex}
We construct an avoidable set $A$ with potential $U(z,A)=\infty$ for symmetric stable processes with $\alpha\in (\frac 2 3,1)$. Then the function $f\coloneqq\1_A$ is a further (bounded) counter example for which $\int_0^\infty f(X_s)\dd s$ is finite with positive probability (it is actually finite almost surely as the tail-$\sigma$-Algebra of a L\'evy process is trivial) but $\int f(x)U(z,dx)=\infty$. The set $A$ is a disjoint union of shrinking intervals drifting off to infinity. The sizes of the intervals are shrinking sufficiently fast as to be avoidable, but not so fast as to have finite potential.

For transient isotropic stable processes a result called Wiener's criterion gives an analytic description of sets which are thin at a point, see \cite{BH} Corollary V.4.17 for a classical presentation. That result can be extended to an analytic description of avoidable sets via \cite{MVisotropic} Corollary A.4, and is stated as follows. With so-called shells $S_n:=\{ x\in\R: \lambda^{n-1} <|x-z|\leq\lambda^{n}\}$ for arbitrary fixed $\lambda\in (1,\infty)$,
\begin{align*}
	B \text{ is }\P_0\text{-avoidable}\quad &\Leftrightarrow\quad \sum_{n=-\infty}^{+\infty} \lambda^{n(\alpha-1)} C(B\cap S_n)<\infty,
%	B \text{ is }\P_z\text{-thin}\quad &\Leftrightarrow\quad\,\,\, \sum_{k=1}^{+\infty} \lambda^{k(\alpha-d)} C(B_k)<\infty
\end{align*}
where $C(B\cap S_n)$ is the capacity of $B\cap S_n$. The capacity of a stable process can be computed or estimated using results from probabilistic potential theory. % Let us fix $\lambda=\frac 1 2$ and define $\tilde S_k=S_{-k}$.
% via 
%$$
	%S_k =  \{ x\in\R : \lambda^{k} < \abs{x - z} \le \lambda^{k+1} \}, \qquad k\in\mathbb Z.
%	\tilde S_k =  \{ x\in\R : 2^{k-1} < \abs{x - z} \le 2^{k} \}, \qquad k\in\mathbb Z,
%$$
%which grow larger as $k$ increases.% This is a reversal of the notation we have seen so far, that is, $\tilde S_k = S_{-k}$, and it makes notation in the following example a little easier. 
Define the sets $A, A_1,\dots \in\mathcal B(\R)$ by
$$
%	A_n \coloneqq \{x\in \R : \gamma^{n}-\gamma^{n/3}\le x < \gamma^{n}\},
	A_n \coloneqq [2^{n}-2^{(n-1)/3}, 2^n) %\{x\in \R : 2^{n}-2^{(n-1)/3}\le x < 2^{n}\},
	\qquad\text{and}\qquad
	A\coloneqq  \bigcup_{n=1}^\infty A_n.
$$
Because $2^{n}-2^{(n-1)/3} > 2^{n}-2^{(n-1)} = 2^{n-1}$, we can take $\lambda=2$ in the definition of $S_n$ and have that $A_n\subseteq S_n$ for each $n \in \mathbb N$, and in particular that $A \cap S_{n} = A_n$. We now show that $A$ is $\P_0$-avoidable for all $\alpha\in(0,1)$.

Let us use the notation
$
	\textbf B_\eps = \{ x\in\R : \abs{x} < \eps\}
$
for the ball around 0 of radius $\eps$. Notice that each $A \cap S_{n}$ is just a translation of the ball of radius $2^{-1}2^{(n-1)/3} = 2^{(n-4)/3}$.
Then Proposition 42.12 of Sato \cite{Sato} it yields that
$
	C(A \cap S_{n})
%	& = C(B_{\gamma^{n/3}/2}) \tag{Sato \cite{Sato} Prop 41.12(vi)} \\
%	& = (\gamma^{n/3}/2)^{1-\alpha} C (B_1). \tag{Sato \cite{Sato} Example 42.17} \\
	= C(\textbf B_{2^{(n-4)/3}})
	= C({2^{(n-4)/3}}\textbf B_1).
$
From the scaling property $C(aB)=a^{1-\alpha}C(B)$  for capacities of stable processes it follows that
$$
	C(A \cap S_{n}) = 2^{(n-4)(1-\alpha)/3} C(\textbf B_1).
$$
Now the summation test of Wiener with $\lambda=2$ is
\begin{align*}
	\sum_{n=-\infty}^\infty2^{n(\alpha-1)}C(A \cap S_{n})
	& = \sum_{n=1}^\infty 2^{n(\alpha-1)}C (A \cap S_{n}) \\
	& = C (\textbf B_1) \sum_{n=1}^\infty  2^{n(\alpha-1)} 2^{(n-4)(1-\alpha)/3} \\
	& = 2^{4(\alpha-1)/3}C (\textbf B_1) \sum_{n=1}^\infty 2^{(2n/3)(\alpha-1)}.
\end{align*}
Since $C (\textbf B_1)=\frac{\Gamma(\frac{1}{2})}{\Gamma(\frac{\alpha}{2})\Gamma((\frac{1-\alpha}{2})+1)}$ is a finite constant - see Example V.4.16(2) of Bliedtner and Hansen \cite{BH}- the quantity above is finite if and only if the geometric series $\sum_{n=1}^\infty 2^{(2n/3)(\alpha-1)}$ is finite, and this is equivalent to the condition $2^{(2/3)(\alpha-1)}<1$, that is, $\alpha <1$. Hence, $A$ is an avoidable set for all stable processes with $\alpha\in (0,1)$.

With the explicit form of the potential we can compute the potential $U(0,A)$:
%We next show that for $2/3 < \alpha < 1$ then $U(0, A) = \infty$.
\begin{align*}
	U(0,A) = \int_A |x|^{\alpha-1}\dd x
	& = \sum_{n=1}^\infty \int_{2^n-2^{(n-1)/3}}^{2^n} x^{\alpha-1} \dd x \\
	& \ge \sum_{n=1}^\infty 2^{n(\alpha-1)} (2^n - (2^n-2^{(n-1)/3})) \tag{monotonicity of $x^{\alpha-1}$} \\
	& = 2^{-1/3}\sum_{n=1}^\infty 2^{n(\alpha-2/3)}.
\end{align*}
This sum is infinite if $\alpha>2/3$. Combining both computations we found an avoidable set $A$ with infinite potential for symmetric stable processes of index $\alpha\in (\frac 2 3,1)$. Setting $f=\1_A$ we have an example function for which $\int_0^\infty f(X_s)\dd s=0<\infty$ with positive probability and $\int_\R f(x)U(0,dx)=\infty$. Hence, the supportive set $M=\R\backslash A$ in $\int_M f(x)U(0,dx)$ is needed also for bounded $f$. The example of course allows modifications for which $f$ has support on $A$ and $f$ is smooth or $f$ vanishes (slowly enough) at infinity.
\end{example}

\section{Finite Time Path Integral Tests}\label{sec:finiteperpetual}

%------------------------------------------------------------------------------

The results of Section \ref{sec:perpetualintegral} for infinite-time-horizon path integrals will now be used to study finite-time-horizon path integrals simultaneously for all finite times. We profit here from the very general assumptions of Theorem \ref{integral test} which allows us to use path integrals in different setups, such as for killed processes. We derive conditions under which
\begin{align*}
	\int_0^t f(X_s) \dd s < \infty\text{ for every }t< \zeta
%\intertext{or}
%	\int_0^t f(X_s) \dd s < \infty\text{ for some }t>0
\end{align*}
holds with positive probability or probability $1$. These questions were raised and answered for the Brownian motion by Engelbert and Schmidt \cite{ES} in 1981, in the form of the following zero-one law. For $f:\R\to [0,\infty]$ measurable they proved the equivalence of 
\begin{enumerate}
	\item $\P_0\big(\int_0^t f(W_s) \dd s < \infty\text{ for every }t\ge 0 \big) > 0;$
	\item $\P_0\big(\int_0^t f(W_s) \dd s < \infty\text{ for every }t\ge 0 \big) = 1;$
%	\item There exists a $t_0>0$ such that $\displaystyle\P_0\Big(\int_0^{t_0} f(W_s) \dd s < \infty\Big) = 1$;
	\item $\int_K f(y) \dd y < \infty$ for all compact $K\subseteq\R$.
\end{enumerate}
The proof is a direct consequence of the occupation time formula $\int_0^t f(B_s)\dd s=\int_\R f(x) L_t(x)\dd x$ and the joint continuity of the local time in both coordinates.
%Engelbert and Schmidt also note that existence of a $t_0>0$ with $\P_0(\int_0^{t_0} f(W_s) \dd s < \infty) > 0$ is not equivalent to  \textbf{(i)} - \textbf{(iv)}.
In one dimension the Brownian motion shares many properties with stable processes of index $\alpha\in(1,2)$, and Zanzotto \cite{Zanzotto97} extended Engelbert and Schmidt's equivalences (i)-(iii) to this class. Zanzotto's proof is identical to the Brownian case as the local time for the stable processes is also jointly continuous. With the same proof the equivalence (i)-(iii) holds equally for all L\'evy processes with jointly continuous local time. As in the previous section the same result cannot hold for $\alpha\in (0,1)$ due to the transient nature of those stable processes, which allows regions with large $f$ values to be avoided, so that avoidable (or supportive) sets must occur in the theorems.
%The theorem above also has a related local version. This local version is worth mentioning because it provides a hope of unifying the story of transient and recurrent processes, since those properties don't much affect the small-time behaviour of a process $X$. 
%\begin{lemma*}[Zanzotto \cite{Zanzotto97} Lemma 1.6]\label{Zanzotto1 local}
%Let $X$ be a stable process of index $\alpha\in(1,2]$ on a probability space $(\Omega, \mathcal F, \P)$, and let $f:\R\to[0,\infty]$, $f\in\mathcal B(\R)$. For a fixed $x\in\R$ suppose there exists a random time $\tau$ such that
%$$
%	\P_x(0<\tau<\infty) = 1 \qquad \text{and} \qquad \P_x\Big( \int_0^\tau f(X_s)\dd s < \infty \Big) >0.
%$$
%Then there exists an $\eps>0$ such that
%$$
%	\int_{-\eps}^\eps f(x+y) \dd y < \infty.
%$$
%\end{lemma*}
%This statement is easily shown to also hold in the reverse direction by assuming the integral test and taking $\tau = T_{\R^d \setminus B_\eps}$. %not true
The results presented in this section are Engelbert-Schmidt-type results for general transient Markov processes. In Section \ref{TMP} we prove theorems which are close to the Engelbert-Schmidt theorem with the exception of the appearance of supportive sets. In Section \ref{sec:LevyLocal} we use our methods to extend the Engelbert-Schmidt/Zanzotto theorem to L\'evy processes with local time which are not necessarily jointly continuous.

%------------------------------------------------------------------------------

\subsection{Transient Markov Processes}\label{TMP}

For general transient Markov processes the Engelbert-Schmidt zero-one law mentioned above fails, since,
%the zero-one law fails as 
with positive probability, $X$ can avoid regions in which $f$ can be arbitrarily large. This in particular implies that a general result should not be a zero-one law at all. Here we present two theorems that together provide a counterpart to Engelbert and Schmidt's theorem for standard transient Markov process. The first concerns the positive probability case, and the theorem for almost sure finiteness is given below.
\begin{theorem}\label{ES2}
Let $X$ be a strongly transient standard Markov process and $f:E\to[0,\infty]$ measurable. For $z\in E$, the following are equivalent:
\begin{enumerate}
	\item $\P_z\big(\int_0^t f(X_s) \dd s < \infty\text{ for every }t< \zeta\big)>0;$
	\item There exists a $\P_z$-supportive set $M$ such that
	$
		\int_{M \cap K} f(y) \,U(z,\ddd y) < \infty
	$
	for all compact $K\subseteq E$.
\end{enumerate}
\end{theorem}
%For one argument in the proof of Theorem \ref{thin ES} it is important that $\tau$ is a stopping time. But if $X$ has the property that for all $t>0$ there $\P_z$-almost surely exists a finely open set $G$ containing $z$ such that $L_G < t$, then $\tau$ need only be a random variable. This condition holds for example for transient Lévy processes that don't hit points, because the last exit times of $\textbf B_\eps(z)$ are arbitrarily small. 
\begin{proof}\hfill\\
 \textbf{(i)$\Rightarrow$(ii)}: 
 Let $(K_n, n\ge1)$ be a sequence of increasing compact sets with last exit times $L_{K_n}$ satisfying $\bigcup_{n\ge1} K_n = E$, and define $f_n = f\1_{K_n}$. Since $X$ is strongly transient, it holds that $\P_z(L_{K_n}<\zeta)=1$ for all $n$, and thus
\begin{align}\label{eq4.d}
\begin{split}
	& \P_z\Big(\int_0^t f(X_s) \dd s < \infty\text{ for every }t< \zeta\Big) > 0 \\ 
	\Rightarrow \quad & \P_z\Big(\int_0^{L_{K_n}} f(X_s) \dd s < \infty\text{ for every }n\in\mathbb N\Big) > 0 \\
	\Rightarrow \quad & \exists p>0 \text{ s.t. }  \P_z\Big(\int_0^\infty f_n(X_s) \dd s < \infty\text{ for every }n\in\mathbb N\Big) > p.
\end{split}
\end{align}
It follows that there exists a constant $C_1>0$ such that
$$
	\P_z\Big( \int_0^\infty f_1(X_s) \dd s< C_1,\, \int_0^\infty f_n(X_s) \dd s < \infty\text{ for every }n\ge2\Big) > p.
$$
Repeating this process inductively for every $n$ gives a sequence of positive, non-decreasing $(C_n, n\ge1)$ such that
$
	\P_z\big( \bigcap_{n=1}^k B_n \big) > p
$
for every $k\ge 1$, where $B_n = \{\int_0^\infty f_n(X_s) \dd s < C_n\}$. Now letting $B \coloneqq \bigcap_{n=1}^\infty B_n$ we see from continuity of measures that
\begin{align}\label{eqab}
	\P_z(B) = \lim_{k\to\infty} \P_z\Big( \bigcap_{n=1}^k B_n\Big) \ge p.
\end{align}
This motivates us to define the Borel set
$
	M \coloneqq \{ y \in E : \P_y(B) > p/2\},
$
which \eqref{eqab} tells us contains the point $z$, and the remainder of the proof will be showing that $M$ is $\P_z$-supportive and satisfies the integral test.

We begin by showing that $M$ is $\P_z$-supportive, which requires first showing that $M^c = E \setminus M$ contains its regular points, a property known as being finely closed. We argue in the spirit of the reasoning used in the proof of Lemma \ref{regular points}. Let $y$ be an arbitrary point which is regular for $M^c$, and so satisfies $\P_y(T_{M^c} = 0)=1$. Let $K\subseteq M^c$ be a compact set, so that, by right-continuity of paths, $X_{T_K}\in K$ almost surely. Then
\begin{align*}
	\P_y(B)
	& = \P_y\Big(\int_0^\infty f_1(X_s) \dd s  <C_1, \dots ;\, T_{K}<\zeta \Big) 
	+ \P_y\Big( \int_0^\infty f_1(X_s) \dd s  <C_1, \dots ;\, T_{K}\ge\zeta \Big) \\
	& \le  \P_y\Big(\int_{T_K}^\infty f_1(X_s) \dd s  <C_1, \int_{T_K}^\infty f_1(X_s) \dd s  <C_2, \dots ;\, T_{K}<\zeta \Big) 
	+ \P_y(T_{K}\ge\zeta ) \\
	& =  \int \P_a\Big(\int_0^\infty f_1(X_s) \dd s  <C_1, \dots \Big) \P_y(X_{T_K}\in \dd a;\, T_{K}<\zeta ) 
	+ \P_y(T_{K}\ge\zeta ) \\
	& =  \int \P_a(B) \P_y(X_{T_K}\in \dd a;\, T_{K}<\zeta )
	+ \P_y(T_{K}\ge\zeta )\\
	& \le \frac p 2 + \P_y(T_{K}\ge\zeta)
\end{align*}
since $K\subseteq M^c$. It follows exactly as in the proof of Lemma \ref{regular points} that $\P_y(T_{K}\ge\zeta) = 0$ and therefore that $y\in M^c$. Thus $M^c$ contains its regular points. Now we can show that $M$ is supportive for any point contained in it. Let $x\in E$ be arbitrary and suppose first that $M$ is \emph{not} $\P_x$-supportive, that is, 
$
	\P_x(D_{M^c}<\zeta)=1.
$
From this it follows that 
\begin{align}\label{eq11}
\begin{split}
	\P_x(B)
	&= \P_x\Big(\int_0^\infty f_1(X_s) \dd s  <C_1, \dots;\, D_{M^c}<\zeta\Big) \\
	&\le \P_x\Big(\int_{ D_{M^c}}^\infty f_1(X_s) \dd s  <C_1, \dots;\, D_{M^c}<\zeta\Big) \\
	&=\int_E \P_a(B) \P_x(X_{D_{M^c}}\in\ddd a ;\,D_{M^c}<\zeta).
\end{split}
\end{align}
We demonstrated above that $M^c$ contains its regular points, and it follows from Blumenthal and Getoor Lemma I(11.4) that $\P_x(X_{D_{M^c}}\in\dd a ;\,D_{M^c}<\zeta)$ is concentrated on $M^c$. In particular, it follows from \eqref{eq11} that
$$
	\P_x(B)
	\le \int_E \P_a(B) \P_x(X_{D_{M^c}}\in\ddd a ;\,D_{M^c}<\zeta) 
	\le \frac p 2.
$$
What we have thus proven is that for arbitrary $x\in E$, if $M$ is not $\P_x$-supportive then $\P_x(B) \le \frac p 2$.
%$$
%	M \text{ is not }\P_x\text{-supportive}
%	\quad \Rightarrow\quad 
%	\P_x(B) \le \frac p 2
%$$
Equivalently,
$$
	\P_x(B) > \frac p 2
	\quad \Rightarrow \quad 
	M \text{ is }\P_x\text{-supportive}.
$$
We saw earlier in the proof that our particular choice of $z\in E$ is contained in $M$, that is $\P_z(B) > \frac p 2$, and thus it follows that $M$ is $\P_z$-supportive. Finally, we show that $M$ satisfies the integral test of (ii). It is clear that for any $n\ge1$,
$$
	M \subseteq M_n \coloneqq \{y\in E : \P_y(B_n) > p/2\} = \Big\{y\in E : \P_y\Big(\int_0^\infty f_n(X_s)\dd s < C_n\Big) > p/2\Big\} .
$$
Since $M_n$ is super-finite for $(X,f_n)$, it follows from Proposition \ref{Getoor++} that $\int_{M_n} f_n(x) \,U(y,\ddd x)  < \infty$ for all $y\in E$. We therefore see that
$$
	\int_{M\cap K_n} f(x) \,U(z,\ddd x) \le \int_{M_n} f_n(x) \,U(z,\ddd x) < \infty \quad \text{for all }n\ge1.
$$
Because $K_n\uparrow E$ any compact set will be covered by $K_n$ for some $n\in\N$. Hence, we obtain the integral test
$\int_{M\cap K} f(x) \,U(z,\ddd x) < \infty$ for all compact $K\subseteq E$.

\textbf{(i)$\Leftarrow$(ii)}:
Suppose the existence of a $\P_z$-supportive $M$ such that for all compact $K$,
$$
	\E_z\Big[ \int_0^\infty f(X_s)\1_{M \cap K}(X_s) \dd s \Big] = \int_{M\cap K} f(y) \,U(z, \ddd y) < \infty.
$$
Since $M$ is supportive there exists a positive $c>0$ such that $\P_z(T_{M^c} \ge \zeta) > c$. Let $K_1, K_2, \dots$ be an increasing sequence of compact sets with limit $E$. Then for all $n\in\mathbb N$,
\begin{align*}
	\E_z\Big[ \int_0^\infty f(X_s)\1_{M \cap K_n}(X_s) \dd s \Big] < \infty
	\quad & \Rightarrow \quad \P_z\Big( \int_0^\infty f(X_s)\1_{M \cap K_n}(X_s) \dd s < \infty \Big) = 1 \\
	\quad & \Rightarrow \quad \P_z\Big( \int_0^\infty f(X_s)\1_{K_n}(X_s) \dd s < \infty \Big) > c \\
	\quad & \Rightarrow \quad \P_z\Big( \int_0^{T_{E\setminus K_n}} f(X_s) \dd s < \infty \Big) > c
\end{align*}
so that
\begin{align}\label{eq4.h}
\begin{split}
	c & \le \lim_{n\to\infty} \P_z\Big( \int_0^{T_{E\setminus K_n}} f(X_s) \dd s < \infty \Big)  = \P_z\Big( \int_0^{T_{E\setminus K_n}} f(X_s) \dd s < \infty \text{ for all } n\in\mathbb N \Big).
\end{split}
\end{align}
Since $X$ is strongly transient,  $\lim_{n\to\infty}T_{E\setminus K_n} = \zeta$, and thus
$$
	\P_z\Big( \int_0^t f(X_s) \dd s < \infty \text{ for all } t< \zeta\Big) \ge c > 0.
$$
The proof is complete.
\end{proof}
In general almost nothing is known about the supportive - or equivalently, avoidable - set in the statement of Theorem \ref{ES2}, but under some tight restrictions on $X$, some results exist. For example, in two recent papers Mimica and Vondraček studied unavoidable unions of balls in $\R^d$, for censored stable processes in \cite{MVcensored} and for isotropic Lévy processes satisfying a particular scaling condition - which generalises the scaling of stable processes - in \cite{MVisotropic}. In a similar vein, Grzywny and Kwaśnicki \cite{GK} give an explicit form of the hitting probability of a ball by a unimodal isotropic Lévy process. However a general characterisation of avoidable sets for general Lévy processes remains a challenging problem. The final example of Hawkes seminal paper \cite{Hawkes} gives a nice demonstration of how the problem differs from that of polar sets.
Nevertheless we will later discuss stronger assumptions on $f$, with $X$ a stable process, under which the supportive sets can be omitted entirely from the Theorem \ref{ES2}.

%It would be desirable to remove the supportive set, but in general this is not possible as can be seen again with the counter example from Example \ref{example}. We will later discuss situations with stable processes in which the supportive sets can be omitted with stronger assumptions on $f$.

The following theorem addresses the same problem as Theorem \ref{ES2} but in the case that the infinite-time path integral is finite almost surely. As in Theorem \ref{ASIT0}, the supportive sets do not disappear, but they can be made large.
\begin{theorem}\label{ES3}
Let $X$ be a strongly transient standard Markov process and $f:E\to [0,\infty]$ measurable. For $z\in E$, the following are equivalent:
\begin{enumerate}
	\item $\P_z\big(\int_0^t f(X_s) \dd s < \infty\text{ for every }t< \zeta\big)=1;$
	\item For every $\eps\in (0,1)$ there exists a $\P_z$-supportive set $M$ such that
	$
		\int_{M \cap K} f(y) \,U(z,\ddd y) < \infty
	$
	holds for all compact $K\subseteq E$ and $X$ stays in $M$ with probability at least $1-\varepsilon$.	
\end{enumerate}
\end{theorem}
\begin{proof}\hfill \\
\textbf{(i)$\Rightarrow$(ii)}: 
Let $K_1 \subseteq K_2 \dots$ be an increasing sequence of compact sets with limit $E$. For $n\in\mathbb N$ let $L_{K_n}$ be the last exit time from $K_n$. Then, since $X$ is strongly transient,
\begin{align}\label{eq4.e}
\begin{split}
	& \P_z\Big(\int_0^t f(X_s) \dd s < \infty\text{ for every }t< \zeta\Big) = 1 \\
	\Rightarrow \quad & \P_z\Big(\int_0^{L_{K_n}} f(X_s) \dd s < \infty\text{ for every }n\in\mathbb N\Big) =1 \\
	\Rightarrow \quad & \P_z\Big(\int_0^{L_{K_n}} f(X_s) \dd s < \infty\Big) =1 \text{ for every }n\in\mathbb N \\
	\Rightarrow \quad & \P_z\Big(\int_0^\infty f(X_s)\1_{K_n}(X_s) \dd s < \infty\Big) =1 \text{ for every }n\in\mathbb N.
\end{split}
\end{align}
Now fix $\eps\in (0,1)$ and $n\in\mathbb N$. We can choose a constant $N_n$ such that 
$$
	\P_z\Big(\int_0^\infty f(X_s)\1_{K_n}(X_s) \dd s \le N_n\Big) > 1-2^{-n}\eps.
$$
Then write $f =f\1_{K_n}$ and define a super-finite set for each $(X,f_n )$,
$$
	M^\eps_n \coloneqq \Big\{ y \in E : \P_y\Big(\int_0^\infty f_n(X_s) \dd s \le N_n\Big) > 2^{-n}\eps \Big\}, \quad n\in\mathbb N.
$$
By Proposition \ref{super finite existence} each $M^\eps_n$ is $\P_z$-supportive. Moreover, using that by Lemma \ref{regular points} $E\setminus M^\eps_n$ contains it's regular points,
\begin{align*}
	 1-2^{-n}\eps
	& < \P_z\Big(\int_0^\infty f_n(X_s) \dd s \le N_n\Big) \\
	& = \P_z\Big(\int_0^\infty f_n(X_s) \dd s \le N_n;\; T_{E\setminus M^\eps_n} < \infty \Big) \\
	& \qquad + \P_z\Big(\int_0^\infty f_n(X_s) \dd s \le N_n;\; T_{E\setminus M^\eps_n} = \infty \Big) \\
	& \le \P_z\Big(\int_{T_{E\setminus M^\eps_n}}^\infty f_n(X_s) \dd s \le N_n;\; T_{E\setminus M_n^\eps} < \infty \Big) + \P_z(T_{E\setminus M_n^\eps} = \infty ) \\
	& =\int_{E\setminus M_n^\eps} \P_a\Big(\int_0^\infty f_n(X_s) \dd s \le N_n\Big) \dd\P_z\Big(X_{T_{E\setminus M_n^\eps}} \in da;\; T_{E\setminus M_n^\eps} < \infty \Big) + \P_z(T_{E\setminus M_n^\eps} = \infty ) \\
	& \le 2^{-n}\eps + \P_z(T_{E\setminus M^\eps_n} = \infty ),
\end{align*}
and therefore not only is each $M_n^\eps$ $\P_z$-supportive, but the probability of remaining in each is bounded away from 0 by
$
	\P_z(T_{E\setminus M^\eps_n} = \infty) > 1-2^{1-n}\eps.
$
Now let $M^\eps \coloneqq \bigcap_n M^\eps_n.$ Then by sub-additivity of measure
\begin{align*}
	\P_z(T_{E\setminus M^\eps} = \infty)
	  \ge 1- \sum_{n=1}^\infty \P_z(T_{E\setminus M^\eps_n} < \infty) 
	  > 1- \eps \sum_{n=1}^\infty 2^{1-n} 
	 = 1- 2\eps.
\end{align*}
Because $K_n\uparrow E$ any compact set will be covered by $K_n$ for some $n\in\N$. Therefore
$$
	\int_{M^\eps\cap K} f(y) \,U(z,\ddd y) \le \int_{M_n^\eps\cap K_n} f(y) \,U(z,\ddd y)= \int_{M_n^\eps} f_n(y) \,U(z,\ddd y).
$$
The righthand side is finite due to Proposition \ref{Getoor++} and the definition of $M_n^\eps$.

\textbf{(i)$\Leftarrow$(ii)}: 
Suppose that for any $\eps>0$ there exists a $\P_z$-supportive set $M^\eps$ so that $X$ stays in $M^\eps$ with probability at least $1-\eps$ and
$
	\int_{M^\eps \cap K} f(y) \,U(z, \ddd y) < \infty
$
for all compact $K\subseteq E$.
Let $K_1 \subseteq K_2, \dots$ be a nested sequence of compact sets with limit $E$.
Then for all $n\in\mathbb N$,
\begin{align*}
	\E_z\Big[ \int_0^\infty f(X_s)\1_{M^\eps \cap K_n}(X_s) \dd s \Big] < \infty
	\quad & \Rightarrow \quad \P_z\Big( \int_0^\infty f(X_s)\1_{M^\eps \cap K_n}(X_s) \dd s < \infty \Big) = 1 \\
	\quad & \Rightarrow \quad \P_z\Big( \int_0^\infty f(X_s)\1_{K_n}(X_s) \dd s < \infty \Big) > 1-\eps \\
	\quad & \Rightarrow \quad \P_z\Big( \int_0^{T_{E\setminus K_n}} f(X_s) \dd s < \infty \Big) > 1-\eps.
\end{align*}
Thus, since $X$ is strongly transient, $\lim_{n\to\infty}T_{E\setminus K_n} = \zeta$ and 
\begin{align}\label{eq4.i}
\begin{split}
	1-\eps & < \lim_{n\to\infty} \P_z\Big( \int_0^{T_{E\setminus K_n}} f(X_s) \dd s < \infty \Big) \\
	& = \P_z\Big( \int_0^{T_{E\setminus K_n}} f(X_s) \dd s < \infty \text{ for all } n\in\mathbb N \Big) \\
	& = \P_z\Big( \int_0^t f(X_s) \dd s < \infty \text{ for all } t< \zeta\Big).
\end{split}
\end{align}
Since $\eps>0$ is arbitrary we see that
$
	\P_z\big( \int_0^t f(X_s) \dd s < \infty \text{ for all } t< \zeta\big) = 1.
$
\end{proof}

%------------------------------------------------------------------------------

\subsection{L\'evy processes with local time}\label{sec:LevyLocal}
One can ask whether, for the special case of L\'evy processes with jointly continuous local time, the additional supportive set compared to the Engelbert-Schmidt \cite{ES} and Zanzotto \cite{Zanzotto97} theorems discussed at the beginning of Section \ref{sec:finiteperpetual} can be removed. Let us quickly recall from Bertoin \cite{Bertoin} some important results on local time for L\'evy processes. If $\Psi$ denotes the characteristic exponent of a L\'evy process $X$, then $X$ has local time if and only if 
\begin{align*}
	\int_{-\infty}^{+\infty} \mathcal R\left(\frac{1}{1+\Psi(\xi)}\right) \dd \xi<\infty,
\end{align*}
see V.1 Theorem 1. A deep result due to Barlow \cite{Barlow} and Barlow and Hawkes \cite{BarlowHawkes} gives a necessary and sufficient condition for the existence of a jointly continuous version of the local time process $(t,x)\mapsto L_t^x$, see Chapter V of \cite{Bertoin}. We now provide a proof of the Engelbert-Schmidt/Zanzotto result without appealing to local time, instead working only with densities of the potential measures of the killed L\'evy process. Our method also works for L\'evy processes for which no jointly continuous version of the local time exists.
\begin{theorem}\label{ES4}
Let $X$ be a Lévy process on $\R$ which has local time, and $f:\R \to [0,\infty]$ be measurable. The following are equivalent:
\begin{enumerate}
	\item\label{ES4i} $\P_0\big(\int_0^t f(X_s) \dd s < \infty\text{ for every }t\ge 0\big)=1;$
	\item	 $\int_{K} f(y) \mathrm d y < \infty$ for all compact $K\subseteq\R$.
\end{enumerate}
%If in addition $X$ is point recurrent and has jointly continuous local time then \textbf{(i)} and \textbf{(ii)} are also equivalent to
%\begin{enumerate}[resume]
%	\item $\displaystyle \P_0\Big(\int_0^t f(X_s) \dd s < \infty\text{ for every }t\ge 0\Big)>0$.
%\end{enumerate}
\end{theorem}
\begin{proof} 
The proof makes use of the killed process $X^q$, that is the L\'evy process killed at an independent exponentially distributed time $\tau^q$ with parameter $q>0$. In the literature such a process is sometimes called a $q$-subprocess, see for example Blumenthal and Getoor \cite{BG} Example III(3.17). The killed process is a standard Markov process on $\R$ with cemetary state $\Delta$, and has transition semigroup
\begin{align}\label{killed semigroup}
	P^q_t(x, A) = \P_x(X_t\in A ; t < \tau^q) = \ee^{-qt}P_t(x, A)
\end{align}
for a Borel set $A$. The potential operator of $X^q$ is $U^q$, the $q$-potential operator of $X$. By Theorem II.5 and Theorem 16 of \cite{Bertoin}, $U^q$ has a bounded density $u^q$. We denote by $\hat{u}^q$ the dual potential density, that is, the potential density of the dual L\'evy process $\hat X=-X$.
 %It is worth noting that for $T$ a stopping time,
%\begin{align}\label{killed semigroup 2}
%	\P_x(T < \tau^q) = \int \P_x(t < \tau^q) \P_x(T\in \dd t) = \int \ee^{-qt} \P_x(T\in \dd t) =  \E_x[ \ee^{-qT}].
%\end{align}
Since $\tau^q$ is almost surely finite, any killed L\'evy process is transient. Since $X$ has lifetime $\zeta=\infty$ almost surely, $X^q$ is not strongly transient.
%It is this which allows us to make use of our theorems for transient processes, even though $X$ itself need not be transient at all.

\textbf{(i)$\Rightarrow$(ii)}: 
Fix $q>0$, and let $X^q$ be the killed process above. Since $\tau^q$ has support $(0,\infty)$ our assumption implies that $X^q$ satisfies condition (i) of Theorem \ref{ASIT0}. Fix some $\eps>0$; by Theorem \ref{ASIT0}, there exists a $\P_0$-supportive set $M^\eps$ for $X^q$ such that $\P_0(X_s \in M^\eps \text{ for all } s<\tau^q)=\P_0(X^q_s \in M^\eps \text{ for all } s<\zeta) > 1-\eps$ and
$$
	\int_{M^\eps \cap K} f(y) U^q(0, \mathrm d y) \le \int_{M^\eps} f(y) U^q(0, \mathrm d y) < \infty
$$
for any compact $K\subseteq\R$. Let us use the notation $B^\eps = \R\setminus M^\eps$ and fix a compact $K\subseteq\R$. Then
\begin{align*}
	\P_0(X_s \in B^\eps\cap K \text{ for some } s<\tau^q) 
	& = \P_0(X^q_s \in B^\eps\cap K \text{ for some } s<\zeta) \\
	& \le \P_0(X^q_s \in B^\eps \text{ for some } s<\zeta) 
	 < \eps.
\end{align*}
Now suppose $B^\eps\cap K$ is non-empty, so there exists some $x\in B^\eps\cap K$. With $T_{\{x\}} = \inf\{s>0 : X_s=x\}$ we find that
\begin{align*}
%	\P_0(X^q_s \in B^\eps\cap K \text{ for some } s<\zeta)
	\P_0(X_s \in B^\eps\cap K \text{ for some } s<\tau^q) \ge \P_0(X_s = x \text{ for some } s<\tau^q) = \E_0[\ee^{-qT_{\{x\}}}] = \hat u^q(x) C^q,
\end{align*}
for a positive constant $C^q$ (using Theorem 43.3 of  \cite{Sato}). Because $X$ hits points, $\hat u^q$ is bounded below on the compact set $K$ (because $u^q$ is lower-semicontinuous as a $q$-excessive function, and point-wise positive). Hence, there is a $c>0$ with
$
	\P_0(X_s \in B^\eps\cap K \text{ for some } s<\tau^q)>c,
$
and this constant $c$ is independent of $\eps$. Now we have shown that $c < \eps$,
%$$
%	c < \P_0(X_s \in B^\eps\cap K \text{ for some } s<\tau^q) < \eps.
%$$
but our choice of $\eps$ was arbitrary, and 
%we can therefore choose $0<\eps< c$. The 
the resolution of this apparent contradiction is that in the case $\eps\le c$ there does not exist any point $x\in B^\eps\cap K$, that is, $K\subseteq M^\eps$. Thus for such $\eps$,
$$
	\infty > \int_{M^\eps \cap K} f(y) U^q(0, \mathrm d y) = \int_K f(y) u^q(y) \dd y.
$$
Again because $u^q$ is bounded below on compacts, $u^q$ can be omitted in the integral test and the claim follows. 

\textbf{(i)$\Leftarrow$(ii)}: 
Fix $q>0$. Since $X$ has local time Bertoin \cite{Bertoin} Theorem II.16 gives that $u^q$ is bounded, and thus for arbitrary compact $K$,
$$
	\E_0\Big[ \int_0^\infty f(X^q_s)\1_{K}(X^q_s) \dd s \Big] = \int_K f(y) u^q(y) \dd y \le \sup_{x\in\R} u^q(x) \int_K f(y) \dd y < \infty.
$$
Now $\E_0\big[ \int_0^\infty f(X^q_s)\1_{K}(X^q_s) \dd s \big] < \infty$ implies $\P_0\big( \int_0^\infty f(X^q_s)\1_{K}(X^q_s) \dd s < \infty \big) = 1$ for all compact $K$. Thus by continuity of measure and monotone convergence
$$
	\P_0\Big( \int_0^{\tau^q} f(X_s) \dd s < \infty \Big)=\P_0\Big( \int_0^\infty f(X^q_s) \dd s < \infty \Big) = 1.
$$
%This holds for all $q$, and so again by continuity of measure
%$$
%	\P_0\Big( \int_0^\infty f(X^q_s) \dd s < \infty \text{ for all } q>0\Big) = 1,
%$$
%which is equivalent to 
Since $\tau^q$ is independent of $X$ and has support $(0,\infty)$, this implies
$$
	\P_0\Big(\int_0^t f(X_s) \dd s < \infty\text{ for every }t< \infty\Big)=1,
$$
and (ii) has been proven.
\end{proof}

\section{Small Time Path Integral Tests}\label{sec:finiteperpetualtwo}
We now come to local versions of the path integral tests. Integral tests for path integrals up to finite random times are known by the work of Engelbert and Schmidt for the Brownian case and Zanzotto for the stable case with $\alpha\in (1,2)$, and take the form of the following equivalent statements:
\begin{enumerate}
\item $\P_z\big(\exists t>0: \int_0^t f(X_s)\dd s < \infty \big) >0;$
\item $\P_z\big(\exists t>0: \int_0^t f(X_s)\dd s < \infty \big) = 1;$
\item There exists an $\eps>0$ such that the integral test $\int_{z-\eps}^{z+\eps} f(y) \dd y < \infty$ holds.
\end{enumerate}
Those authors' arguments, as in the case of finite time-horizon path integrals, rely upon the occupation time formula and the joint continuity of local time. For our study of stable SDEs for $\alpha\in (0,1)$ we need a version for transient L\'evy processes. A more general version for transient Markov processes can be found in the Ph.D thesis Baguley \cite{Samphd}.

\begin{theorem}\label{thin ES}
Let $X$ be a transient Lévy process on $\R^d$ which does not hit points and $f:\R^d\to[0,\infty]$ measurable. For $z\in \R^d$, the following are equivalent:
\begin{enumerate}
	\item $\P_z\big(\exists t>0: \int_0^t f(X_s)\dd s < \infty \big) = 1;$	
	\item There exists a $\P_z$-thin set $B$ such that $\int_{\R^d\setminus B} f(y) \,U(z,\ddd y) < \infty$. 
\end{enumerate}
Recall that a measurable set $B$ is called $\P_z$-thin if $\P_z(T_B>0)=1$.
\end{theorem}
The reader should keep in mind the following relation of avoidable (or supportive) sets and thin sets. Since an avoidable set is never hit with positive probability the first hitting time is trivially non-zero with positive probability. Since this event obeys a zero-one law (Blumenthal-Getoor zero-one law) the first hitting time is non-zero with probability one. Hence, any $\P_z$-avoidable set is also $\P_z$-thin. The converse is obviously wrong.

\begin{proof}
Note that $(i)$ is equivalent to 
\begin{enumerate}
	\item[$(i)^\prime$] There exists a finite non-zero random time $\tau$ such that $\P_z\big( \int_0^\tau f(X_s)\dd s < \infty \big) = 1$.
\end{enumerate}
We will use $(i)^\prime$ with first and last hitting times to deduce the theorem from the theorems of the previous section.

$\boldsymbol{(i)^\prime\Rightarrow(ii)}$:
%%
%Now suppose that $\P_z(T_{E\setminus \{z\}}>0) =0$.
%%
%We shall prove that there exists a $\P_z$-thin set $B\in\mathcal E$ such that $z\in E\setminus B$ and $L_{E\setminus B} \le \tau$ with positive probability. 
%
%Since $X$ doesn't hit points, $\P_z(L_{\{z\}} = 0)=1$ where $L_{z}$ is the last hitting time of $z$. 
Since $X$ does not hit points, i. e. $\P_z(L_{\{z\}} = 0)=1$, it follows from quasi-left-continuity of $X$ that the last exit times of balls $\textbf B_\eps(z)$ are arbitrarily small as $\eps\downarrow 0$. In particular, there exists some open set $G$ such that $z\in G$ and 
$
	\P_z(L_G \le \tau)>0.
$
Thus under the assumption of (i'),
\begin{align*}
	\P_z\Big( \int_0^{L_G} f(X_s)\dd s < \infty \Big) 
	& \ge \P_z\Big( \int_0^{L_G} f(X_s)\dd s < \infty ; \tau > L_G\Big) \\
	&\ge \P_z\Big( \int_0^\tau f(X_s)\dd s < \infty ; \tau > L_G \Big)  >0.
\end{align*}
It therefore follows that
\begin{align}
	\P_z\Big( \int_0^\infty f(X_s)\1_G(X_s)\dd s < \infty \Big) >0.
\end{align}
Then Theorem \ref{integral test} yields a $\P_z$-supportive set $M$ such that
$$
	\int_{M\cap G} f(x) \, U(z,\ddd x) < \infty.
$$
Since $M$ is $\P_z$-supportive, the complement of $\R^d\backslash M$ is avoidable so that $\P_z(T_{\R^d\backslash M}>0)\geq \P_z(T_{\R^d\backslash M}=\infty)>0.$ Hence, Blumenthal's zero-one law implies that $\P_z(T_{\R^d\backslash M}>0)=1$, that is, $\R^d \setminus M$ is $\P_z$-thin. Since $G$ is open and contains $z$, $\R^d\setminus G$ is also $\P_z$-thin. The union of finitely many $\P_z$-thin sets is again $\P_z$-thin, and thus we define $B = (\R^d \setminus M) \cup (\R^d\setminus G) = \R^d\setminus (M\cap G)$, which is $\P_z$-thin and satisfies
$$
	\int_{\R^d\setminus B} f(x) \, U(z,\ddd x)=\int_{M\cap G} f(x) \, U(z,\ddd x) < \infty.
$$

$\boldsymbol{(i)^\prime\Leftarrow(ii)}$: 
Let $B$ be the $\P_z$-thin set from (ii). Proposition II(4.3) of Blumenthal and Getoor \cite{BG} says that there is a compact set $K\subseteq \R^d\setminus B$ such that $z\in K$ and $\R^d\setminus K$ is again $\P_z$-thin. The random time we shall define is
$
	\tau = T_{\R^d\setminus K}.
$
Since $\R^d\setminus K$ is $\P_z$-thin, $\tau$ is $\P_z$-almost surely positive. In addition by transience of $X$, $T_{\R^d\setminus K} \le L_K < \infty$ $\P_z$-almost surely, and so $\P_z(0<\tau<\infty) = 1$. We have assumed that
$$
\E_z\Big[ \int_0^\infty f(X_s) \1_{\R^d\setminus B}(X_s) \dd s \Big]=\int_{\R^d\setminus B} f(x) \, U(z,\ddd x) < \infty.
$$
This implies that $\P_z( \int_0^\infty f(X_s)\1_{\R^d\setminus B}(X_s)\dd s < \infty) = 1$. From $K\subseteq \R^d\setminus B$ it follows that $\tau \le T_B$. Therefore
\begin{align*}
	\P_z\Big( \int_0^\tau f(X_s)\dd s < \infty \Big) 
	&\ge \P_z\Big( \int_0^{T_B} f(X_s)\dd s < \infty \Big) \\
	&\ge \P_z\Big( \int_0^\infty f(X_s)\1_{\R^d\setminus B}(X_s)\dd s < \infty \Big) = 1.
\end{align*}
\end{proof}

\subsection{The Case of Stable L\'evy Processes}

It is generally impossible to remove the thin sets from Theorem \ref{thin ES}. The following theorem is a version of Theorem \ref{thin ES} in a particular situation in which the thin set can be removed by capacity comparisons. Its proof relies on a remarkably precise analytic description of $\P_z$-thin sets for stable processes. The theorem will be applied later with $f=\sigma^{-\alpha}$ to study the SDE \eqref{eq1}, for instance with $\sigma(x)=|x|^\beta$. 
\begin{theorem}\label{ES stable}
Let $X$ be a symmetric stable process on $\R$ with index $\alpha\in(0,1)$, and let $f:\R\to [0,+\infty]$ be measurable. Suppose that $f$ has an isolated monotone pole at $z$, in the sense that there exists $\delta>0$ such that $f$ is monotone increasing on $(z-\delta,z)$, monotone decreasing on $(z,z+\delta)$. Then the following are equivalent:
\begin{enumerate}
	\item $\P_z\big(\exists t>0: \int_0^t f(X_s)\dd s < \infty \big) = 1;$	
	\item There exists an $\eps>0$ such that $ \int_{z-\eps}^{z+\eps} f(y) \abs{z-y}^{\alpha-1}\dd y < \infty$. 
\end{enumerate}
\end{theorem}

Theorem \ref{ES stable} above can be directly extended to `almost monotone' measurable functions $g$, in the sense that there exists a $C<\infty$ such that for all $\abs{x-z}<\delta$, $\abs{g(x)-f(x)}\le C$ for some measurable $f$ which has an isolated monotone pole at $z$, by virtue of the fact that in this case, for $\eps<\delta$,
$$
	\left| \int_{z-\eps}^{z+\eps} g(y) \abs{y}^{\alpha-1}\dd y - \int_{z-\eps}^{z+\eps} f(y) \abs{y}^{\alpha-1} \dd y\right| \le \frac{2C\eps^\alpha} \alpha \quad \text{ and }\quad
	\left| \int_0^t g(X_s) \dd s - \int_0^t f(X_s) \dd y\right|\le Ct.
$$
\begin{proof} 
Since $X$ is a Lévy process we can without loss of generality set $z=0$. Recall that $U(0,\ddd y) = \abs{y}^{\alpha-1} \dd y$, up to a constant factor which we freely ignore. The implication (ii)$\Rightarrow$(i) follows from Theorem \ref{thin ES} and the fact that $\R\setminus(-\eps,\eps)$ is $\P_0$-thin for the stable process.

Now suppose (i). Again from Theorem \ref{thin ES} there exists a $\P_0$-thin set $B$ such that 
$$
	\int_{\R\setminus B} f(y) \,U(0,\ddd y) = \int_{\R\setminus B} f(y) \abs{y}^{\alpha-1} \dd y< \infty.
$$
Take $\eps\in(0,\delta)$ and let the map $g$ be defined by $g(y) = f(y) \abs{y}^{\alpha-1}\1_{(-\eps,\eps)}(y)$. It is immediately seen that $g$ shares the same monotonicity property as $f$. The intuition to have in mind is that  the monotone nature of $g$ will allow its behaviour on $B$ to be determined by its behaviour on $\R\setminus B$.

According to Wiener's Criterion for thin sets of stable processes (see for instance Corollary V.4.17 of Bliedtner and Hansen \cite{BH}), the $\P_0$-thin set $B$ satisfies
\begin{align}\label{eq4.a}
	\sum_{k=1}^\infty 2^{k(1-\alpha)}C(B\cap S_k) < \infty,
\end{align}
where $C(B\cap S_k)$ is the capacity of $B\cap S_k$ and $S_k = \{x\in\R : 2^{-(k+1)} < \abs{x} \le 2^{-k} \}$ defines a sequence of decreasing shells of Lebesgue measure $2(2^{-k}-2^{-(k+1)}) = 2^{-k}$. The isoperimetric inequality of Betsakos \cite{Betsakos} Theorem 1 states that the capacity of any compact set is greater or equal that of the ball of the same Lebesgue measure. Example 42.17 of Sato \cite{Sato} then gives that for any compact $K$,
$$
	C(K) \ge C(\textbf B_{\frac 1 2\lambda(K)}) 
	= C\Big(\frac {\lambda(K)} 2 \textbf B_1\Big) 
	= \Big(\frac {\lambda(K)} 2\Big)^{1-\alpha} C(\textbf B_1)
	= C_0\lambda(K)^{1-\alpha},
$$
where $\textbf B_r$ is the ball about $0$ of radius $r$ and $C_0\coloneqq C(\textbf B_1)2^{\alpha-1}$. We can extend this result to a wider class of Borel sets as follows. Let $G\subset \R$ be a bounded open set. Then there exists an increasing sequence of compact sets $K_n\subseteq G$ such that $G = \bigcup_{n=1}^\infty K_n$, and thus from Sato \cite{Sato} Propositions 42.10 and 42.12 it follows that $C(K_n)\uparrow C(G)$. The isoperimetric inequality applied to each $K_n$ then yields that
$$
	C(G) = \lim_{n\to\infty}C(K_n) \ge \lim_{n\to\infty}C_0\lambda(K_n)^{1-\alpha} = C_0\lambda(G)^{1-\alpha} .
$$
Now for a general bounded Borel set $A\subset \R$,
\begin{align*}
	C(A) 
	& = \inf\{C(G) : \text{$G$ open and $A\subseteq G$}\} \\
	& \ge C_0 \inf\{\lambda(G)^{1-\alpha}  : \text{$G$ open, bounded and $A\subseteq G$}\}
	= C_0 \lambda(A)^{1-\alpha}.
\end{align*}
The first equality is due to $C$ being a Choquet capacity, and details can be found in I(10.5) of \cite{BG} and the discussion following it. This in particular holds for the bounded sets $B\cap S_k$, and so from \eqref{eq4.a} we can deduce that
$$
	\sum_{k=1}^\infty 2^{k(1-\alpha)}C_0\lambda(B\cap S_k)^{1-\alpha} 
	= C_0 \sum_{k=1}^\infty (2^k \lambda(B\cap S_k))^{1-\alpha} 
	< \infty.
$$
Since $C_0= C(\textbf B_1)2^{\alpha-1}<\infty$ and $ \lambda(S_k) = 2^{-k}$ this implies 
\begin{align*}%\label{eq3}
	\sum_{k=1}^\infty \Big(\frac {\lambda(B \cap S_k) }{ \lambda(S_k) }\Big)^{1-\alpha}< \infty.
\end{align*}
From this convergent sum it follows that for any fixed $c\in(0,1)$ there exists $N\in\mathbb N$ such that for all $n\ge N$, $\lambda(B \cap S_n)\le c\lambda(S_n) = c2^{-n}$, and therefore that 
\begin{align}\label{eq4.1c}
	\lambda(B^c\cap S_n)\ge (1-c)\lambda(S_n) = (1-c)2^{-n}.
\end{align}
We will use this relationship to bound the integral of $g$ over $B$.

For $n\geq N$ we shall now consider the two pieces of $S_n$ separately, using notation $S_n^+ = S_n \cap (0,\infty)$, $S_n^- = S_n \cap (-\infty,0)$. Taking advantage of the monotonicity of $g$, with the notation $\bar g_n = \sup_{S^+_n} g$ and $\underline g_n = \inf_{S^+_n} g$, we see for $n\ge N$ that
\begin{align*}
	\int_{B \cap S^+_n} g(x) \dd x
	& \le \bar g_n \lambda(B \cap S^+_n)
	\le \bar g_n c\lambda(S^+_n)
	= \bar g_n 2c\lambda(S^+_{n+1}).
\intertext{Using \eqref{eq4.1c} and the fact that $\bar g_n \le \underline g_{n+1}$, we continue the chain of inequalities as}
	& \le \frac{2c}{1-c} \; \underline g_{n+1}  \lambda(B^c\cap S^+_{n+1})
	\le \frac{2c}{1-c}\int_{B^c\cap S^+_{n+1}} g(x) \dd x.
\end{align*}
Exactly the same procedure works for $S^-_n$, and adding the two pieces gives
$$
	\int_{B \cap S_n} g(x) \dd x \le \frac {2c}{1-c} \int_{B^c\cap S_{n+1}} g(x) \dd x.
$$
Summing over $n\ge N$ tells us that
$$
	\int_{B\cap \textbf B_{2^{-N}}} g(x) \dd x \le \frac {2c}{1-c} \int_{B^c} g(x) \dd x < \infty.
$$
Let $\tilde \eps = \eps \wedge 2^{-N}$. Summing the integrals over $B$ and $B^c$ then yields $\int_{-\tilde \eps}^{\tilde\eps} f(x) \abs{y}^{\alpha - 1} \dd x < \infty$.
\end{proof}

\section{Stochastic Differential Equations}\label{sec:SDE}
In this final section we translate the foregoing results on different path integrals into results for stable stochastic differential equations. The (driftless) \emph{stable SDE equation} with issuing point $z\in\R$ is defined to be the equation
\begin{align}\label{SDEeq}
	\dd Z_t = \sigma (Z_{t-}) \dd X_t, \qquad Z_0 = z.
\end{align}
More precisely, let $X$ be a symmetric stable process on $\R$ and probability space $\mathscr P = (\Omega,\mathcal F,\P)$, and $Z$ an $\R$-valued stochastic process on the same probability space satisfying $\P(Z_0=z)=1$ for some $z\in\R$. Let $\sigma:\R\to[0,\infty)$ be a measurable function, extended as is usual via $\sigma(\Delta)=0$. For $A\subseteq \R$ the collection $(X,Z,\mathscr P)$ is called a \emph{weak solution} to \eqref{SDEeq} on the measurable set $A$ if 
\begin{align}\label{SDE}
	Z_t - z = \int_0^t \sigma (Z_{s-}) \dd X_s \qquad \text{for all } t< T_{A^c}
\end{align}
and $T_{A^c}>0$. The latter condition only excludes trivial cases. If $A=\R$ then $(X,Z,\mathscr P)$ is called a \emph{global weak solution}, and if $A\subsetneq \R$ then $(X,Z,\mathscr P)$ is called a \emph{local weak solution}. We say a solution is trivial if it is constant.

\subsection{Existence and Uniqueness of (Local) Solutions}
Following the classical Engelbert-Schmidt approach to Brownian SDEs, in a sequence of articles Zanzotto succeeded in reformulating weak solutions of $\dd Z_t=\sigma(Z_{t-})\dd X_t$ into time-changes of the driving process with the path integrals $\int_0^t \sigma^{-\alpha}(X_s)\dd s$. Using quadratic variations this is straightforward for the Brownian motion (see for instance Chapter 5 of Karatzas and Shreve \cite{KS}), but the arguments for the stable case are more involved. The following reciprocal connection of SDE solutions and time-change is a combination of three of Zanzotto's results (Lemma 2.26 of \cite{Zanzotto97}, Theorem 2 of \cite{Zanzotto98}, and Theorem 2.2 of \cite{Zanzotto02}) and a time-change due to Kallenberg (Theorem 4.1 of \cite{Kallenberg}).
\begin{theorem}[Zanzotto/Kallenberg time-change]\label{Z time change}\hfill\\
\textbf{(i)} Let $X$ be a symmetric stable process of index $\alpha\in(0,2]$ on the probability space $(\Omega, \mathcal F, \P_z)$, and let $\sigma:\R\to [0,\infty)$ be measurable. Define
$$ 
	I_t = \int_0^t \sigma(X_s)^{-\alpha} \dd s, \qquad  
	\varphi_t = \inf\Big\{ s > 0 : \int_0^s \sigma(X_u)^{-\alpha} \dd u > t \Big\}, \qquad t\ge0,
$$
%$I_t = \int_0^t \sigma(X_s)^{-\alpha} \dd s$ and $\varphi_t = \inf\{ s>0 : I_s>t\}$,
and let $X_\varphi$ be the time-changed process with the definitions from \eqref{time change} and below. Then there exists a symmetric stable process $Y$ of index $\alpha$ on an extension $\overline{\mathscr P} = (\overline\Omega, \overline{\mathcal F}, \overline\P)$ of $(\Omega, \mathcal F, \P_z)$ such that
\begin{align}\label{SDEeq1}
\begin{split}
	& X_{\varphi_t} - z = \int_0^t \sigma(X_{\varphi_s}) 
	%\1_{(\varphi_s < \bar\varphi)}
	\dd Y_s, \qquad t \in [0, \bar I).
\end{split}
\end{align}
\textbf{(ii)} Let $(X,Z,\mathscr P)$ be a global weak solution of \eqref{SDEeq} with initial condition $z$.
Then there exists a symmetric stable process $Y$ of index $\alpha\in (0,2]$ on an extension $\overline{\mathscr P} = (\overline\Omega, \overline{\mathcal F}, \overline\P)$ of $(\Omega, \mathcal F, \P)$ such that
\begin{align}\label{eq5.5}
	%Y_t = Z_{\tilde \varphi_t}, \qquad \text{$t\in[0,\tilde{I}_{T_{\R\backslash A}})$},
	Z_t = Y_{\tilde I_t} \qquad \text{for all $t\in[0,\infty)$ $\overline\P$-a.s.},
\end{align}
where
$
	\tilde I_t = \int_0^t \sigma(Z_{s})^{\alpha} \dd s.
	%, \qquad \tilde \varphi_t = \inf\{ s > 0 : \tilde I_s > t \},  \qquad t\geq 0.
$
\end{theorem}
Zanzotto's abstract time-change representation in \eqref{SDEeq1} was proved for all $\alpha\in (0,2)$, whereas the translation into SDEs (i.e. showing the running time is a first hitting time of a known set $A$) and the precise analysis of situations in which the time-change representation is non-trivial was restricted to $\alpha\in (1,2)$. This was caused by the lack of understanding of finiteness of path integrals which for $\alpha\in (0,1)$ we fully provided in the sections above. The usefulness of our precise results lies not only in providing necessary and sufficient properties of $\sigma$ for the existence and uniqueness of solutions but also in giving information about properties of those solutions, for example whether they are global, and whether they explode in finite time. In what follows we shall present the complete picture for stable SDEs with $\alpha\in (0,1)$, utilising the integral tests from Sections \ref{sec:perpetualintegral}-\ref{sec:finiteperpetualtwo}.

%We shall use the Theorems developed in Chapter 4 to prove counterparts to these for stable processes on $\R$ of index $\alpha\in(0,1)$. 
%In his original proof Zanzotto made use of the following result, given in Kallenberg \cite{Kallenberg} Theorem 3.1 and the discussion above and below it.
%\begin{lemma}\label{kallenberg lemma}
%For $\sigma \in \mathcal B(\R)_+$, $X$ a stable process of index $\alpha$ and $Z$ a predictable\footnote{define? needed? mabye not since I take $s-$} process on $(\Omega, \mathcal F, \P)$, and 
%any $t\in[0,\infty)$,
%$$
%	\int_0^t \sigma(Z_{s-}) \dd X_s < \infty \text{ a.s.}\qquad \text{if and only if} \qquad \int_0^t \sigma(Z_{s-})^\alpha \dd s <  \infty \text{ a.s.}
%$$
%In particular the above hold for all $t< T_{\R\setminus B}$ when \eqref{SDEeq} has a weak solution on $B$.
%\end{lemma}

This first proposition is a more general version of Lemma 2.3 of Zanzotto \cite{Zanzotto02}. It is more involved than the original because of the complexity of finely open sets for general stable processes, compared to the simpler case for $\alpha\in (1,2]$. The lemma connects the statement (i) of Theorem \ref{Z time change} to the notion of local weak solutions, by demonstrating that the running time $\bar I$ of $X_\varphi$ in \eqref{SDEeq1} is indeed a first hitting time of a measurable set.

\begin{proposition}\label{O lemma}
Let $X$ be a Markov process on $E$ with strong Feller resolvent satisfying Hunt's condition (see \ref{Hunt condition} at the end of Section \ref{sec:setting}), and let $f$ be non-negative and measurable. Define the path integrals and inverses as before:
$$
	I_t = \int_0^t f(X_s) \dd s, \qquad  
	\varphi_t = \inf\Big\{ s > 0 : \int_0^s f(X_u) \dd u > t \Big\}, \qquad t\in[0,\infty).
$$
Then $T_\mathcal{O} = \varphi_\infty$ $\P_y$-almost surely for all $y\in E$, where
$$
	 \mathcal{O} = \Big\{x\in\R^d : \P_x\Big(\forall t>0:\int_0^t f(X_s) \dd s = \infty\Big) = 1 \Big\}.
$$
If $f$ is strictly positive, then $I(T_\mathcal{O}) = \bar I = \inf\{s>0: X_{\varphi_s}\in \mathcal O\}$.
\end{proposition}
Note that this statement does not hold for all Markov processes, or even all stable processes. It fails for instance for $X$ a deterministic positive drift on $\R$ with $f(x)=|x|^{-1} \1_{(-\infty,0)}(x)$, because in that case $\mathcal O=\emptyset$ but $\varphi_\infty = T_{\{0\}}< \infty$ under $\P_x$ if the initial value $x$ is negative. However the proposition does apply for two-sided stable L\'evy processes.
\begin{proof}
Fix a $q>0$ and let $X^q$ be the process killed at an independent exponential time $\tau^q$, so that
$$
	\int_0^\infty f(X_s^q) \dd s = \int_0^{\tau^q} f(X_s) \dd s \qquad \text{almost surely}.
$$
For $c>0$ we define the sets
$
	M_{c} \coloneqq \big\{ x\in\R :  \P_x\big( \int_0^\infty f(X^q_s) \dd s \le \frac 1 c\big) > c \big\}
$
and $A_c \coloneqq \R\setminus M_c.$
The sets $M_c$ are super-finite for $(X^q,f)$, and increasing as $c\downarrow0$. Their complements $A_c$ decrease as $c\downarrow 0$ to
$
	A :=\big\{ x\in\R :  \P_x\big( \int_0^\infty f(X^q_s) \dd s < \infty \big) = 0 \big\}.
$
It is clearly true that $\mathcal O \subseteq A$. We can also note that if $x\in A$ then
$$
	\P_x\Big(\int_0^t f(X_s) \dd s = \infty \quad \text{for all }t\geq 0\Big)=1,
$$
since $\tau^q$ is independent and has support $(0,\infty)$. Therefore $x\in\mathcal O$, and so $A=\mathcal O$. For our purposes it is easier to work with $A$ than $\mathcal O$ for the rest of the proof.

We saw in Lemma \ref{super finite technical lemma 2} that for any choice of $c$, it holds that for all $y\in\R$
\begin{align}\label{eq5.12}
	\P_y\Big(\int_0^t f(X_s^q) \1_{M_{c}}(X^q_s) \dd s < \infty \quad \text{for all }t\ge0\Big)=1.
\end{align}
Since $X^q$ stays in $M_c$ before $T_{A_c}$ this implies that $\P_y$-almost surely $\int_0^t f(X_s^q) \dd s< \infty$ for all $t \le T_{A_c}$, where $T_{A_c}$ is the first exit time of $M_{c}$ by the unkilled process $X$. Since the sets $A_c$ are decreasing the times $T_{A_{c}}$ are almost surely increasing. In particular, since \eqref{eq5.12} holds for arbitrary $c>0$, it follows that, for any $y\in\R$,
$$
	\P_y\Big(\int_0^t f(X_s^q) \dd s< \infty \quad \text{ for all }t < \lim_{c\downarrow0} T_{A_{c}}\Big)=1.
$$
Again, since $\tau^q$ is independent of $X$ and has support $(0,\infty)$, this implies that
$$
	\P_y\Big(\int_0^t f(X_s) \dd s< \infty \quad\text{ for all }t < \lim_{c\downarrow0} T_{A_{c}}\Big)=1.
$$
Therefore it follows that  $\P_y$-almost surely $\varphi_\infty \ge \lim_{c\downarrow0} T_{A_{c}}$. If we can show the first equality of
\begin{align}\label{eq5.11}
	T \coloneqq \lim_{c\downarrow0} T_{A_{c}} = T_A=T_{\mathcal O}
\end{align}
$\P_y$-almost surely for any $y\in\R$ then we will have proven the first inequality $\varphi_\infty\geq T_{\mathcal O}$.

Recall that for a Borel set $A$ and any $q>0$ the function $\Phi^q_A(x) = \E_y[\ee^{-q T_A}; T_A < \infty] = \E_y[\ee^{-q T_A}]$ is $q$-excessive. Since our process $X$ satisfies condition \ref{Hunt condition} and $\Phi^q_A$ is bounded, the discussion at the end of Section \ref{sec:setting} tells us that $\Phi^q_A$ is regular (in the sense of \ref{Hunt regularity} or \ref{Hunt regularity alt}), and thus in particular that
$$
	\Phi^q_A(X_{T_{A_{c}}}) \to \Phi^q_A(X_T) \qquad \text{almost surely on $\{T<\infty\}$ as $c\downarrow0$.}
$$
If we fix $A = A_{c_0}$ for some $c_0>0$ then $\Phi^q_A(X_{T_{A_{c}}})$ equals one for all $c\le c_0$, since the sets $A_c$ are decreasing and contain their regular points by Lemma \ref{regular points} so that $X_{T_{A_c}}$ is contained in $A_c\subseteq A_{c_0}$ on the event $\{T<\infty\}$ for any $y\in \R$. Thus the limit $\Phi^q_A(X_T)=1$ on $\{T<\infty\}$, that is, $\P_{X_T(w)}(T_{A_{c_0}}=0)=1$ for $\P_y$-almost every $w$ such that $T(w) < \infty$. Since $A_{c_0}$ contains its regular points, this implies that $X_T\in A_{c_0}$ on $\{T<\infty\}$. Then because our choice of $c_0>0$ was arbitrary, it follows that
$$
	X_T \in A =  \bigcap_{c>0} A_{c} \qquad \text{almost surely on }\{T<\infty\}.
$$
This implies \eqref{eq5.11} on $\{T<\infty\}$. On the event $\{T=\infty\}$, \eqref{eq5.11} is trivial, and so it holds almost surely, and thus $\varphi_\infty \ge T_A = T_{\mathcal O}$ almost surely.

The inequality $\varphi_\infty \le T_{\mathcal O}$ comes from the fact that for any $u>0$, any $y\in\R$,
\begin{align*}
	\P_y(\varphi_\infty \le T_{\mathcal O} + u\; ; T_{\mathcal O} < \infty)
	& = \P_y\Big(\int_0^{T_{\mathcal O} + u} f(X_s) \dd s = \infty\; ; T_{\mathcal O} < \infty\Big) \\
	& = \E_y\Big[ \P_{X_{T_{\mathcal O}}}\Big(\int_0^u f(X_s) \dd s = \infty\Big) \; ; T_{\mathcal O} < \infty\Big] \\
	& = \P_y(T_{\mathcal O} < \infty),
\end{align*}
using that $\mathcal O=A=\cap_{c>0} A_c$ and that all $A_c$ (and thus the intersection) are finely closed by Lemma \ref{regular points} so that $X_{T_{\mathcal O}}\in \mathcal O$. This implies that $\P_y(\varphi_\infty \le T_{\mathcal O} + u) = \P_y(\varphi_\infty \le T_{\mathcal O} + u < \infty) + \P_y(T_{\mathcal O} = \infty) = 1$. Therefore $\varphi_\infty \le T_{\mathcal O} + u$ almost surely for all $u>0$, and so $\varphi_\infty \le T_{\mathcal O}$ almost surely.

Finally, we note that if $f>0$ then the integrals $t\mapsto I_t$ are strictly increasing (up to a possibly finite explosion time). Hence, the generalised inverse $\varphi$ does not jump to $+\infty$, that is, $\varphi_\infty=\bar \varphi=\varphi_{\bar I}$. Now since $T_\mathcal{O} = \varphi_\infty$ it follows that
$
	\bar I=\inf\{s>0: X_{\varphi_s}\in \mathcal O\}.
$
It remains to note that since $I$ is left-continuous, $\bar I = I(\varphi_\infty) = I(T_\mathcal{O})$.
\end{proof}
Although the setting of Lemma \ref{O lemma} is quite general, we will always use it for stable SDEs with $f=\sigma^{-\alpha}$. Since the presence of the set $\mathcal O$ is crucial (local solutions will live on the complement) and it depends on $\alpha$ and $\sigma$ we shall give it a name:
\begin{definition}\label{O def}
For $\sigma:\R\to [0,\infty)$ measurable and $X$ a symmetric stable L\'evy process we denote the set of irregular points by
\begin{align*}
	 \mathcal{O}(\sigma,\alpha) & := \Big\{x\in\R : \P_x\Big(\int_0^t \sigma(X_s)^{-\alpha} \dd s = \infty\Big) = 1
	  \text{ for all }t> 0 \Big\}
\end{align*}
and the null-set of $\sigma$ by $N(\sigma) := \{x\in\R : \sigma(x)=0\}$.
\end{definition}
Note that the definition of $\mathcal O(\sigma,\alpha)$ is purely stochastic and as such is not very useful when used in SDE theorems. It is the results on path integrals which will make the following time-change results for SDEs useful.\smallskip

Before coming to the main theorem let us first translate the first part of Theorem \ref{Z time change} into a statement on (local) weak solutions on the complement of the irregular points
\begin{proposition}\label{sde cor0}
If $z\in \R\backslash \mathcal O(\sigma,\alpha)$, then $(Y, X_\varphi,\overline{\mathscr P})$ from Theorem \ref{Z time change}(i) is a local weak solution to the SDE \eqref{SDEeq} on $A=\R\setminus \mathcal{O}(\sigma,\alpha)$. In addition, if either 
\begin{enumerate}
	\item[\namedlabel{SDE A1}{$(a)$}] $\P_z(\sigma(X_{\varphi_\infty})=0)=1$, or
	\item[\namedlabel{SDE A2}{$(b)$}] $\P_z(\bar I=\infty)=1$
\end{enumerate}
hold, then the solution is a global solution.
%If $z\in \R\setminus \mathcal{O}(\sigma,\alpha)$ then $(Y, X_\varphi,\overline{\mathscr P})$ is non-trivial.
\end{proposition}

\begin{proof}
Since $\sigma^{-\alpha}>0$,
%the integrals $t\mapsto I_t$ are strictly increasing (up to a possibly finite explosion time). Hence, the generalised inverse does not jump to $+\infty$, that is $\varphi_\infty=\bar \varphi=\varphi_{\bar I}$. 
Proposition \ref{O lemma} tells us that
$
	\bar I=\inf\{s>0: X_{\varphi_s}\in \mathcal O(\sigma,\alpha)\}.
$
Thus, Zanzotto's time-change implies that with $Z:=X_{\varphi}$ the triple $(Y,Z,\overline{\mathscr P})$ is a local solution to \eqref{SDEeq} on the complement of $\mathcal O(\sigma,\alpha)$. Assuming either \ref{SDE A1} or \ref{SDE A2} implies that the integral equation \eqref{SDEeq1} holds for all $t\in[0,\infty)$, and therefore that the solution is global. 
\end{proof}
%\begin{corollary}\label{sde cor1} DELETE
%If either
%\begin{enumerate}
%	\item[(A1)] $\sigma(X_{\varphi_{\infty}}) = 0$ $\P_z$-almost surely, or
%	\item[(A2)] $\bar I = \infty$ $\P_z$-almost surely,
%\end{enumerate}
%then $(Y, X_\varphi,\overline{\mathscr P})$ is a weak solution to \eqref{SDEeq}.
%\end{corollary}
%\begin{proof}
%Direct from Theorem \ref{Z time change}. 
%\end{proof}

%\begin{corollary}\label{sde cor2} DELETE FOR NEXT MORE COMPACT ONE
%If
%\begin{enumerate}
%	\item[(A3)] $ \mathcal{O}(\sigma,\alpha) \subseteq N(\sigma)$
%\end{enumerate}
%then (A1) holds for all $z\in\R$, and thus there exists a weak solution $(Y, X_\varphi,\overline{\mathscr P})$ to \eqref{SDEeq} for all $z\in\R$.
%\end{corollary}
%\begin{proof}
%To see that (A3) implies (A1) for all $z\in\R$, recall that Corollary \ref{O corollary}, alongside the fact that $\mathcal{O}(\sigma,\alpha)$ is finely closed and thus contains its regular points, tells us that $X_{\varphi_{\infty}} \in  \mathcal{O}(\sigma,\alpha)$ $\P_z$-almost surely for all $z\in\R$. Then Corollary \ref{sde cor1} yields the result.
%\end{proof}

Conversely, we use the second part of Theorem \ref{Z time change} to derive necessary conditions for the existence of solutions.
\begin{proposition}\label{sde_cor4}
If there exists a non-trivial local weak solution to \eqref{SDEeq} on a set $A$ with issuing point $z$, then $z\in\R\setminus \mathcal{O}(\sigma,\alpha)$. This in particular holds when $A=\R$, that is in the case of global weak solutions.
\end{proposition}
\begin{proof}
 To make use of Kallenberg's time-change let us turn the local into a global solution by defining $\sigma^\dagger \coloneqq \sigma \cdot (1-\1_{A^c\cup(A^c)^r})$. Then the process $Z^\dagger_t \coloneqq Z_{t\wedge T_{A^c}}$ satisfies
$$
	Z^\dagger_t - z = \int_0^t \sigma^\dagger (Z^\dagger_{s-}) \dd X_s, \qquad \text{for all } t\ge0.
$$
Thus  Kallenberg's time-change tells us that there exists a symmetric stable process $Y$ such that $Z^\dagger_t = Y_{\tilde I^\dagger_t}$ for all $t\ge0$ almost surely, where $\tilde I^\dagger_t = \int_0^t \sigma^\dagger(Z^\dagger_s) \dd s$, and thus it follows that $Z_t = Y_{\tilde I_t}$ for all $t\le T_{A^c}$ almost surely, using $\tilde I_t = \int_0^t \sigma(Z_s) \dd s$ and $\tilde \varphi_t = \inf\{ s > 0 : \tilde I_s > t \}$. Since $Z$ is not constant it follows from the SDE equation that $\sigma(Z_s)>0$ for some Lebesgue-positive set of times $s<T_{A^c}$, and thus $\tilde I$ is also not constant zero, and $\tilde \varphi$ does not jump instantaneously to $\infty$. Combining this with
\begin{align*}
	 \tilde \varphi_t
	 \ge \int_0^{\tilde \varphi_t} \1_{(\sigma(Z_s)>0)} \dd s
	 =  \int_0^{\tilde \varphi_t} \sigma(Z_u)^{-\alpha} \sigma(Z_u)^{\alpha} \dd u 
	 =  \int_0^t \sigma(Z_{\tilde \varphi_s})^{-\alpha} \dd s 
	 =  \int_0^t \sigma(Y_s)^{-\alpha} \dd s \eqqcolon I_t
\end{align*}
we have shown that there is almost surely some $t<T_{A^c}$ such that $I_t\leq\tilde \varphi_t<\infty$. It therefore holds by definition of $\mathcal O(\sigma,\alpha)$ that the issuing point $z$ of $Y$ under $\overline\P$ is an element of $\R\setminus  \mathcal{O}(\sigma,\alpha)$.
\end{proof}
Finally, we can use our results on path integrals to turn the abstract formulations into analytic results. In the case that $\alpha\in(1,2]$, Zanzotto provided the analytic expression
\begin{align}\label{O set greater 1}
	 \mathcal{O}(\sigma,\alpha) = \Big\{x\in\R : \int_{x-\eps}^{x+\eps}  \sigma(y)^{-\alpha} \dd y = \infty \text{ for all }\eps>0  \Big\}.
\end{align}
In the case that $\alpha \in (0,1)$ according to Theorem \ref{thin ES} it holds that
\begin{align}\label{O set less 1}
	 \mathcal{O}(\sigma,\alpha) = \Big\{x\in\R : \int_{\R\setminus B} \sigma(y)^{-\alpha} \abs{x-y}^{\alpha-1} \dd y = \infty\text{ for all }\P_x\text{-thin sets } B \Big\}
\end{align}
since $U(x,dy)=|x-y|^{\alpha-1}\dd y$ modulo some normalising constant. Since Wiener's criterion (see \cite{BH} Corollary V.4.17) gives an analytic test for thinness in terms of capacities, the test is also analytic. If in addition $\sigma$ has only isolated monotone zeros (e.g. $\sigma(x)=|x|^\beta$) then Theorem \ref{ES stable} implies for $\alpha\in (0,1)$ the clean integral tests
$$
	 \mathcal{O}(\sigma,\alpha) = \Big\{x\in\R : \int_{x-\eps}^{x+\eps}  \sigma(y)^{-\alpha}|x-y|^{\alpha-1} \dd y = \infty \text{ for all }\eps>0  \Big\},
$$
 which is precisely Zanzotto's integral test modulo an additional polynomial factor. Recall from the introduction that complements of all $\P_z$-thin sets for the stable process with $\alpha>1$ contain a ball around $z$. Hence, on a structural level the difference between the integral tests for $\alpha\in(1,2]$ and $\alpha\in(0,1)$ is the appearance of the new polynomial factor.

We are now in a position to formulate a set of statements which, for the Brownian motion, are known under the name Engelbert-Schmidt theorems. The main theorem of this article extends the Engelbert-Schmidt theorems to stable SDEs with $\alpha\in (0,1)$. This same result was proved for $\alpha\in (1,2)$ by Zanzotto \cite{Zanzotto02}, and our proofs follow his closely, with the important distinction being that our Proposition \ref{O lemma} is a more general version of his Lemma 2.3 that no longer depends on local times.

\begin{theorem}\label{sde conditions} Suppose $X$ is a symmetric stable process with $\alpha\in (0,1)$ and $\sigma:\R\to [0,\infty)$ is measurable. Then the following statements hold with the set of irregular points $\mathcal O(\sigma,\alpha)$ from \eqref{O set less 1}.
\begin{enumerate} 
	\item\label{sde thm1} For fixed $z\in\R$ there exists a non-trivial local weak solution to the SDE \eqref{SDEeq} if and only if $z\in\R\backslash \mathcal O(\sigma,\alpha)$.
	
	\item\label{sde thm2} A global weak solution to \eqref{SDEeq} exists for all $z\in\R$ if and only if $\mathcal{O}(\sigma,\alpha)\subseteq N(\sigma)$.
	
	\item\label{sde thm3} A non-trivial global weak solution to \eqref{SDEeq} exists for all $z\in\R$ if and only if $\mathcal{O}(\sigma,\alpha) = \emptyset$.
	
	\item\label{sde thm4} There exists a global weak solution to \eqref{SDEeq} for all $z\in\R$, each of which is unique in law, if and only if $\mathcal{O}(\sigma,\alpha) = N(\sigma)$. In that case the solution process $Z$ satisfies $Z= Y_{\varphi}$, where $Y$ is a symmetric stable process on $\R$ of index $\alpha$ and $\varphi_t = \inf\big\{ s > 0 : \int_0^s \sigma(Y_u)^{-\alpha} \dd u > t \big\}, t\ge0.$
\end{enumerate}
\end{theorem}
Let us compare again with the case $\alpha\in (1,2]$. In this case $\R\backslash\mathcal O(\sigma,\alpha)$ is open so in the situation of (i) one can always consider local solutions on small intervals $[u,v]$ around the starting value $z$. Without further assumptions on $\sigma$ this is generally false for $\alpha\in (0,1)$ using examples where $\sigma$ vanishes on very small disjoint intervals accumulating at $z$.
\begin{proof}\hfill\\
\ref{sde thm1}:
First suppose that $z\in\R\setminus  \mathcal{O}(\sigma,\alpha)$. From Proposition \ref{sde cor0} it follows that there exists a local weak solution, with solution process $Z = X_\varphi$. Further $T_{\mathcal O(\sigma,\alpha)}>0$ almost surely, as $\mathcal O(\sigma,\alpha)$ is finely closed for $X$ (see the end of the proof of Proposition \ref{O lemma}). Otherwise $z$ would be a regular point of $\mathcal O(\sigma,\alpha)$ and as such in $\mathcal O(\sigma,\alpha)$. This implies that $\varphi_\infty=T_{\mathcal O(\sigma,\alpha)}>0$, and thus that $t\mapsto \int_0^t \sigma^{\alpha}(X_s)\dd s$ does not jump to $+\infty$ immediately. This implies that the time-change does not explode instantaneously (i.e.\ $\bar \varphi>0$) almost surely, and so the solution $Z=X_{\varphi}$ is not trivial.

For the reverse implication, from Proposition \ref{sde_cor4} it follows that if there exists a non-trivial local weak solution then $z\in\R\setminus \mathcal{O}(\sigma,\alpha)$. 

\ref{sde thm2}:
%``$\Leftarrow$":  
Since $\mathcal O(\sigma,\alpha)$ is finely closed for $X$ (see the end of the proof of Proposition \ref{O lemma}) we see that $X_{T_{\mathcal O(\sigma,\alpha)}}\in \mathcal O(\sigma,\alpha)$ on the event that the hitting time is finite. Hence if we suppose that $\mathcal{O}(\sigma,\alpha)\subseteq N(\sigma)$, it follows that $\sigma(X_{\varphi_\infty}) = \sigma(X_{T_{\mathcal O(\sigma,\alpha)}}) = 0$ $\P_z$-almost surely, for any $z\in\R$. Thus Proposition \ref{sde cor0} gives existence of a global solution, with solution process $Z = X_\varphi$.
	
%``$\Rightarrow$": 
Now suppose that a global solution exists for every issuing point $z\in\R$. Proposition \ref{sde_cor4} tells us that if $z \in \mathcal{O}(\sigma,\alpha)$ then there is no non-trivial local weak solution to \eqref{SDEeq} with issuing point $z$. If we then assume that there exists a weak solution for all issuing points, it follows that the solution for $z \in  \mathcal{O}(\sigma,\alpha)$ is trivial, and therefore that $\sigma(z)$ must be zero for $z \in\mathcal{O}(\sigma,\alpha)$.
	
\ref{sde thm3}:
%``$\Leftarrow$"
First suppose that $\mathcal{O}(\sigma,\alpha)$ is empty. Then since $\sigma^{-\alpha}>0$, Proposition \ref{O lemma} yields that $\bar I=\inf\{s>0: X_{\varphi_s}\in \mathcal O\}=\infty$ almost surely under any $\P_z$. Therefore for any $z\in\R$, Proposition \ref{sde cor0} gives existence of a global weak solution, with solution process $Z = X_\varphi$.

%``$\Rightarrow$": 
Now we prove the reverse implication. Proposition \ref{sde_cor4} tells us that if $z \in \mathcal{O}(\sigma,\alpha)$ then there is no non-trivial solution to \eqref{SDEeq} with issuing point $z$. If we then assume that there exists a non-trivial weak solution for all issuing points, it follows that $\mathcal{O}(\sigma,\alpha)$ is empty.

\ref{sde thm4}:
``$\boldsymbol\Rightarrow$" Suppose that for every $z\in\R$ there exists a global weak solution to \eqref{SDEeq} and that each of those solutions is unique in law. Then part (ii) of this theorem implies that $\mathcal{O}(\sigma,\alpha)\subseteq N(\sigma)$. Now suppose for contradiction that there is a point $z\in N(\sigma) \cap (\R\setminus \mathcal{O}(\sigma,\alpha))$. Since $z\in N(\sigma)$, the trivial solution is a solution. Since $z\in\R\backslash \mathcal O(\sigma,\alpha)$, (i) tells us that there exists a non-trivial weak solution with solution process $Z=X_{\varphi}$. Then we have two weak solutions issued from $z$ which are not equal in law, and this contradicts uniqueness. Hence, $N(\sigma)\subseteq \mathcal{O}(\sigma,\alpha)$, and the equality has been proved.

``$\boldsymbol \Leftarrow$" Now suppose that $\mathcal{O}(\sigma,\alpha) = N(\sigma)$. For this part of the proof we will closely follow the proof of Theorem 2.6 in Zanzotto \cite{Zanzotto02}, the crucial difference here being that our Lemma \ref{O lemma} generalises his Lemma 2.3.  
%By (ii), a global weak solution exists for all $z\in\R$. If we fix one of these solutions $(X,Z,\mathscr P)$, then Kallenberg's time-change representation of \cite{Kallenberg} (Theorem 4.1) yields that there is a symmetric stable process $Y$ such that $Y_{\tilde I_t} = Z_t, t\in [0,\infty)$, defined in general on an extension of $\mathscr P$, where $\tilde I_t =  \int_0^t \sigma(Z_{s-})^{\alpha} \dd s$.
By (ii), a global weak solution exists for all $z\in\R$. We now use the representation of Kallenberg in Theorem \ref{Z time change} (ii) to prove the time-change representation for this solution, and thus to deduce uniqueness. If we fix one of these solutions $(X,Z,\mathscr P)$, then Kallenberg's time-change representation of \eqref{eq5.5} yields that there is a symmetric stable process $Y$ such that $Y_{\tilde I_t} = Z_t, t\in [0,\infty)$,  defined in general on an extension of $\mathscr P$, where $\tilde I_t \coloneqq  \int_0^t \sigma(Z_{s})^{\alpha} \dd s$. Since $\tilde I_{\tilde \varphi_t} = t\wedge \tilde I_\infty$ we obtain $Z_{\tilde \varphi_s} = Y_s$ for $s\in[0,\tilde I_\infty)$, where $\tilde \varphi_s \coloneqq \inf\{t>0 : \tilde I_t>s\}$. Further,
%\begin{align*}
% See Zanzotto 2002 equation (2.2)
%	\tilde \varphi_t
%	\ge \int_0^{\tilde \varphi_t} \1_{\{\sigma(Z_u) > 0\}} \dd u
%	= \int_0^{\tilde \varphi_t} \sigma(Z_u)^{-\alpha} \sigma(Z_u)^{\alpha} \dd u
%	= \int_0^{\tilde \varphi_t} \sigma(Z_u)^{-\alpha} \dd \tilde I_u
% Engelbert and Schmidt (1985) Lemma 1.6
%	=  \int_0^{\tilde I(\tilde \varphi_t)} \sigma(Z_{\tilde \varphi_s})^{-\alpha} \dd s.
%\end{align*}
\begin{align}\label{sayso1}
	\tilde \varphi_t \ge \int_0^{\tilde \varphi_t} \1_{\{\sigma(Z_u) > 0\}} \dd u
	= \int_0^{t\wedge \tilde I_\infty} \sigma(Z_{\tilde \varphi_s})^{-\alpha} \dd s
	=  \int_0^{t\wedge \tilde I_\infty} \sigma(Y_s)^{-\alpha} \dd s,
	\quad t\ge0.
\end{align}
We shall now show that \eqref{sayso1} holds as an equality. Let $T$ be the first hitting time of $\mathcal O(\sigma,\alpha)$ by $Y$, and suppose at first that $\tilde I_\infty \le T$. We know that the relation $Y_t = Z_{\tilde \varphi_t}$ holds for all times $t< \tilde I_\infty$, and therefore it follows from the fact that $\tilde \varphi$ is strictly increasing that the first hitting time of $\mathcal O(\sigma,\alpha)$ by $Z$ is greater than or equal $\tilde \varphi_{\tilde I_\infty}$. In particular, since $\mathcal O(\sigma,\alpha)=N(\sigma)$, the indicator in the left-hand integral of \eqref{sayso1} is identically equal $1$ for all $t\ge0$, and thus the equation is an equality. Now suppose that $\tilde I_\infty \ge T$. From the relation $Y_t = Z_{\tilde \varphi_t}$ we see that $\tilde \varphi_T$ is the first hitting time of $\mathcal O(\sigma,\alpha)$ by $Z$, and as above we note that up to this time equality holds in \eqref{sayso1} because the indicator on the left is equal $1$. After time $T$ the right-hand-side of \eqref{sayso1} is equal to $+\infty$ (this is the main result of Proposition \ref{O lemma}, where we showed that $T=\varphi_\infty$) and so we have shown in any case that
\begin{align}\label{sayso2}
	\tilde \varphi_t =  \int_0^{t\wedge \tilde I_\infty} \sigma(Y_s)^{-\alpha} \dd s,
	\quad t\ge0.
\end{align}
But in fact now we can show something more. By definition either $\tilde I_\infty=+\infty$ or $\tilde \varphi_{\tilde I_\infty} = +\infty$. Since $T$ is (again by Proposition \ref{O lemma}) less than or equal any time $t$ for which the integral $I_t = \int_0^t \sigma(Y_s)^{-\alpha} \dd s$ has value infinity, it follows that in either case $\tilde I_\infty \ge T$. But in addition, since $\tilde \varphi$ is by definition strictly increasing, it follows from \eqref{sayso2} that $\int_0^t \sigma(Y_s)^{-\alpha} \dd s$ must be finite for all times less than $\tilde I_\infty$, and so $\tilde I_\infty\le T$. We can deduce that the two times are almost surely equal. It directly follows that 
\begin{align}\label{sayso3}
	\tilde \varphi_t =  \int_0^t \sigma(Y_s)^{-\alpha} \dd s,
	\quad t\ge0,
\end{align}
with both sides being equal $+\infty$ after time $\tilde I_\infty = T$. Plugging this into the relation $Y_{\tilde I_t} = Z_t$ for 
%$t\in[0,\tilde I_\infty)$ yields first that in fact, $Y_t=Z_{\tilde \varphi_t}$ on $t\ge0$, and second 
$t\ge0$ yields
that $Z$ has $\overline \P$-almost sure representation $Z_t = Y_{\varphi_t}$ for $ t\ge 0.$ If $z\in \mathcal{O}(\sigma,\alpha)$ the time-change is identically 0 and the time-change solution and trivial solution coincide, so the solution is unique. If $z\in \R\setminus N(\sigma)$ then the trivial solution does not exist, so the non-trivial time-change solution is unique. Since $\mathcal{O}(\sigma,\alpha) = N(\sigma)$ we have proven uniqueness for all $z\in\R$.
\end{proof}

We can specialize the theorem a bit more if $\sigma$ has a monotone zero. In that case the thin-sets disappear from the integral test. The situation might look artificial but occurs in many examples.\medskip

\begin{corollary}\label{sde thm1b}
Suppose $\sigma$ has an isolated monotone zero at $z\in\R$. Then there exists a non-trivial local weak solution to the SDE \eqref{SDEeq} started in $z$ if and only if 
%there exists a $\P_z$-thin set $B$ such that
%$$
%	\int_{\R\setminus B} \sigma(y)^{-\alpha} \abs{y}^{\alpha-1} \dd y < \infty.
%$$
 $\int_{z-\eps}^{z+\eps}  \sigma(y)^{-\alpha}|z-y|^{\alpha-1} \dd y<\infty$ for some $\eps>0$.
 \end{corollary}
\begin{proof}
Follows from Theorem \ref{sde conditions} (i), where Theorem \ref{ES stable} gives the path integral over $f=\sigma^{-\alpha}$.
\end{proof}
Let us consider the previous corollary for the particular situation 
$$
	\dd Z_t = |Z_{t-}|^\beta \dd X_t, \qquad Z_0 = 0.
$$
According to the corollary there is a non-trivial solution if and only if
$$
	\int_{-\eps}^{\eps} \sigma^{-\alpha}(y) \abs{y}^{\alpha-1} \dd y 
	= \int_{-\eps}^{\eps} \abs{y}^{-\alpha\beta + \alpha-1} \dd y 
	< \infty
$$
for some $\eps>0$, which holds if and only if $-\alpha\beta + \alpha-1 > -1$, that is, $\beta < 1$. Since the trivial solution $Z\equiv 0$ is also a solution, our integral test combined with the known results for $\alpha\in (1,2]$ applied to this situation gives the failure of uniqueness for the simple polynomial stable SDE if and only if $\beta<(1/\alpha)\wedge 1$. This corresponds nicely to the counter examples of Bass, Burdzy and Chen \cite{BBC} for the pathwise uniqueness for stable SDEs.

\subsection{Properties of Solutions}

In this final section we will explore properties of solutions to the stable SDE 
\begin{align*}
	\dd Z_t = \sigma (Z_{t-}) \dd X_t, \qquad Z_0 = z,
\end{align*}
 using the integral tests we developed earlier in this paper. %Let $X$ be a symmetric stable process of index $\alpha\in(0,1)\cup(1,2]$ on $\R$, and let $\sigma:\R\to[0,\infty)$ be a measurable function. Unfortunately, the case $\alpha=1$ eludes classification, at least with the methods presented here. In order to explore properties of solutions to \eqref{SDEeq} we need that those solutions exist, and thus following Theorem \ref{sde conditions} we assume that 
Let us assume $\mathcal{O}(\sigma,\alpha)\subseteq N(\sigma)$, which ensures by Theorem \ref{sde conditions} that there exists a global solution with solution process with the time-change representation $Z=Y_\varphi$. 
%where $Y$ is a symmetric stable process with the same law as $X$, and $\varphi_t \coloneqq \inf\big\{ s>0 : \int_0^s \sigma(Y_u)^{-\alpha}\dd u > t\big\}$. From either of the analytic expressions of $\mathcal{O}(\sigma,\alpha)$, which we gave in \eqref{O set greater 1} and \eqref{O set less 1}, it is immediately seen that if $\sigma$ is continuous then $\mathcal{O}(\sigma,\alpha)\subseteq N(\sigma)$ holds, but we cannot give a more precise analytic description of which $\sigma$ suffice in general. 
%It is worth noting that if we were to assume instead that $\sigma$ is bounded away from zero on compact sets\footnote{This property is satisfied by many functions of interest, including e.g.\ positive lower-semicontinuous functions, which includes positive $q$-excessive functions.} then both $\mathcal{O}(\sigma,\alpha)$ and $N(\sigma)$ would be empty, and Theorem \ref{sde conditions} (iii) and (iv) tells us that a unique, non-trivial, global solution would exist for all $z\in\R$.
 Since $\sigma$ takes values in $[0,\infty)$, it follows that the path integral $I_t = \int_0^t \sigma(Y_s)^{-\alpha}\dd s$ is continuous in $t$, which, as we noted around \eqref{volk2}, ensures that $Z$ is a strong Markov process. Note that the solution can explode in finite time in which case, $Z$ is sent to the cemetery state $\Delta$.
\subsubsection*{Explosion}
Let $\zeta\coloneqq \inf\{t>0:Z_t = \Delta\}$ denote the lifetime of $Z$. Since the driving stable process $Y$ has infinite lifetime, the time-change representation $Z=Y_\varphi$ implies that
$$
	\zeta = \int_0^\infty \sigma(Y_s)^{-\alpha}\dd s \quad \text{ almost surely.}
$$
We say that $Z$ \emph{explodes} if $\zeta<\infty$. If $I_t$ is finite for all $t\ge0$ almost surely then the event $\{\zeta<\infty\}$ is in the tail-$\sigma$-algebra of the stable process $Y$, and so in this case explosion becomes a zero-one-law. In general, however, $Z$ can explode with probability in $(0,1)$. Note that the recurrence of the symmetric stable processes excludes the possibility of finite time explosion for all $\alpha\in (1,2)$. We can now fully characterise explosion through integral tests via our earlier results. Theorem \ref{integral test} gives a necessary and sufficient condition for explosion of $Z$ with positive probability, and Theorem \ref{ASIT0} does the same but for explosion with probability one.% Moreover, in the case that $\alpha\in(1,2]$, Theorem \ref{ES4} gives a sufficient (and necessary) condition for $I$ to be finite everywhere almost surely, which as we noted above implies that explosion is a zero-one law.

Here is an example. If we assume that $\sigma$ is bounded away from zero on compact sets, a unique global weak solutions exist for all initial conditions and according to Theorem \ref{ASIT}. Explosion of $Z$ is a zero-one law, and $Z$ explodes almost surely if and only if there exists a transient set $B$ such that
$$
	\int_{\R\setminus B} \sigma(x)^{-\alpha} \abs{x-z}^{\alpha-1} \dd x < \infty.
$$
%If $\alpha\in(1,2]$ then $Y$ is point recurrent, and the only choice of $B$ is the empty set. In addition in this case the measure $U$ is infinite on all sets with positive Lebesgue measure, and it follows that explosion occurs if and only if $\sigma(x)^{-\alpha}$ is zero almost everywhere. It was an assumption of this section that $\sigma$ takes values in $[0,\infty)$, and so explosion of Z cannot occur.
%If $\alpha\in(0,1)$ then this integral test can be written
%$
%	\int_{\R\setminus B} \sigma(x)^{-\alpha} \abs{x-z}^{\alpha-1} \dd x < \infty.
%$
It would be nice to remove B from the integral test, but Example \ref{ex} shows that in general this is not possible. Under strong regularity assumption the questions has been addressed for all starting conditions simultaneously in Döring and Kyprianou \cite{DK} by appealing to duality theory for Markov processes.
%
%Explosion of SDE solutions was also recently studied for positive continuous $\sigma$ by Döring and Kyprianou \cite{DK}. The same paper also used a theory of time-reversal for Markov processes developed by Nagasawa \cite{Nagasawa} to consider a related problem, called `entrance at infinity' of SDE solutions, in which an SDE solution, expressed as the time-reversal of another Markov process, is well-defined and non-trivial under a law which has it issuing from the infinity point.

\subsubsection*{Freezing}
We say that $Z$ is \emph{frozen} if there exists a time $t\in[0,\infty)$ such that $Z_s = Z_t$ for all $s\ge t$. It follows from the time-change representation $Z=Y_\varphi$ is frozen if and only if $\varphi_\infty < \infty$, which from the definition of $\varphi$ occurs if and only if there exists a $t\in[0,\infty)$ such that
$$
	\int_0^t \sigma(Y_s)^{-\alpha} \dd s = \infty.
$$
Thus it is clear that freezing and explosion preclude one another. If $\alpha\in(1,2)$ then $Y$ is point recurrent, and Zanzotto's zero-one law (or Theorem \ref{ES4}) yields that freezing is a zero-one law, and gives a sufficient and necessary condition for freezing to occur. If $\alpha\in(0,1)$ then $Y$ is transient, and Theorem \ref{ES3} gives a sufficient and necessary condition for $Z$ to be frozen with positive probability, while Theorem \ref{ES2} gives a sufficient and necessary condition for freezing to occur almost surely.

\addcontentsline{toc}{section}{Bibliography}
\bibliography{bibliography}

\end{document}